\documentclass[letter]{article}
\usepackage[top=3.5cm, bottom=3.5cm, outer=3.5cm, inner=3.5cm, heightrounded]{geometry} 
\usepackage[english]{babel}
\usepackage[utf8x]{inputenc}
\usepackage{amsmath}
\usepackage{graphicx}
\usepackage[colorinlistoftodos]{todonotes}
\usepackage{amsmath}
\usepackage{amssymb}
\usepackage{amsfonts}
\usepackage{enumitem}
\usepackage{authblk}
\usepackage{bigints}

\usepackage{xcolor}
\def\Var{{\textrm{Var}}\,}
\def\E{{\textrm{E}}\,}
\def\Cov{{\textrm{Cov}}\,}
\usepackage{kpfonts}
\usepackage{stackengine}
\usepackage{calc}
\usepackage{indentfirst}
\usepackage{hyperref}
\usepackage{bbm}
\usepackage{mathtools}
\DeclarePairedDelimiter{\ceil}{\lceil}{\rceil}
\newlength\shlength
\newcommand{\indep}{\perp \!\!\! \perp}
\usepackage{tikz}
\newcommand*\circled[1]{\tikz[baseline=(char.base)]{
		\node[shape=circle,draw,inner sep=2pt] (char) {#1};}}
\newtheorem{lemma}{Lemma}[section]
\newtheorem{prop}{Proposition}[section]
\newtheorem{thm}{Theorem}[section]
\newenvironment{proof}{\paragraph{Proof:}}{\hfill$\square$}
\usepackage{url}
\makeatletter
\newcommand{\address}[1]{\gdef\@address{#1}}
\newcommand{\email}[1]{\gdef\@email{\url{#1}}}
\newcommand{\@endstuff}{\par\vspace{\baselineskip}\noindent\small
	\begin{tabular}{@{}l}\scshape\@address\\\textit{E-mail address:} \@email\end{tabular}}
\AtEndDocument{\@endstuff}
\makeatother

\begin{document}

\title{The damped wave equation and associated polymer}
\author{Yuanyuan Pan}
\address{Dept. of Mathematics, University of Rochester}
\email{ypan29@ur.rochester.edu}

\maketitle

\begin{abstract}
	Considering the damped wave equation with a Gaussian noise $F$ where $F$ is white in time and has a covariance function depending on spatial variables,  we will see that this equation has a mild solution which is stationary in time $t$. We define a weakly self-avoiding polymer with intrinsic length $J$ associated to this SPDE. Our main result is that the polymer has an effective radius of approximately $J^{5/3}$.
\end{abstract}

\section{Introduction}\label{sec1}

Polymers are studied intensively in many fields. There are many works studying different aspects of polymers. For example,  \textit{Random polymer models} deals with equilibrium statistical mechanics of a class of polymers \cite{Gia07}. \textit{Random polymers} focuses on the interface between probability theory and equilibrium statistical physics \cite{dHo07}.  \textit{The theory of polymer dynamics} concentrates on the dynamics of polymers in the liquid state \cite{DoEd86}. Our work is inspired by Mueller and Neuman's \cite{MN22}. Their work studied the radius of  polymers without self-intersection and showed that, considering the heat equation with white noise, the effective radius of the polymers is approximately $J^{5/3}$.

As stated in \cite{MN22}, the simplest model for a polymer is random walk. If self-intersection is prohibited, we are led to study self-avoiding random walk. We are interested in finding macroscopic extension of a polymer. Such extension is often measured by the variance of the end-to-end distance, $\E[\lvert S_n\rvert^2]$, where $S_n$ is the location of a polymer at $n$ units from its beginning $S_0$. One famous problem is to show that $\E[\lvert S_n\rvert^2]\approx Cn^{2\nu}$ where $(S_n)_{n\in\mathbb{N}}$ is the simple random walk on $\mathbb{Z}^d$ with self-avoiding path, and $\nu$ is a constant depending on $d$.  

\begin{enumerate}
	\item[(1)] When $d\geq 5$, $\nu=\frac{1}{2}$. Hara and Slade followed the idea of Brydges and Spencer \cite{BS85}, and verified the result in \cite{TS91} and \cite{TS92}.
	
	\item[(2)] Almost nothing is known rigorously about $\nu$ in dimension 2, 3, and 4.
	\begin{enumerate}
		\item[(i)] When $d=2$, based on non-rigorous Coulomb gas methods, Nienhuis \cite{NienhuisBernard1982ECPa} predicted that $\nu=\frac{3}{4}$. This predicted value has been confirmed numerically
		by Monte Carlo simulation, e.g. \cite{LMS95}, and exact enumeration of self-avoiding walks up to length n = 71 \cite{JensenIwan2004Eosw}.

		\item[(ii)]  For $d=3$, $\nu$ is expected to be 0.588$\cdots$. An early prediction for the values of $\nu$, referred to as the Flory values \cite{FloryP.J.1951Tcor}, was $\nu=\frac{3}{d+2}$ for $1\leq d\leq 4$. This does give the correct answer for $d=1,2,4$, but it is not accurate when $d=3$. The Flory argument is very remote from a rigorous mathematical proof. 
		
		\item[(iii)] When $d=4$, $\nu$ is expected to be $\frac{1}{2}$. And there should be a logarithmic correction. Dimension four is the upper critical dimension for the self-avoiding walk. The expected number of intersections between two independent random walks tends to infinity, but only logarithmically in the length. Such considerations are related to the logarithmic corrections. Partial works for this case can be found in \cite{BDCGS12} and \cite{BSTW17} and the references therein.
	\end{enumerate}
	
	\item[(3)] When $d=1$, it is obvious to see $\nu=1$.
\end{enumerate}

The case $d=1$ is the simplest, but it still presents challenging questions. For example, if we consider the weakly self-avoiding one-dimensional simple random walks $(S_n)_{n\in\mathbb{N}}$ with $S_0=0$, there is a complete answer \cite{GdH93} to characterize the limiting speed, $\underset{n\rightarrow\infty}{\lim}\frac{1}{n}\left(\E[S_n^2]\right)^{1/2}$. There has also been work on the continuous-time situation, see \cite{HdHK97}.

We  study the radius of polymers that satisfy the damped wave equation in one dimensional space. The wave equation can be used to study the propagation of mechanical waves or vibrational modes within the polymer structure. Polymers are composed of long chains of repeating molecular units, and these chains can exhibit certain vibrational modes when excited. In \cite{DoEd86} chapter 4, if we consider the discrete case, we can use the Rouse model to describe the motion of internal beads of a polymer, that is
\begin{equation}
	dX_i(t)=\Delta X_i(t)dt+dB_i(t) 
\end{equation}\label{rousemodel}
where $\Delta$ is the discrete Laplacian. If $F$ is a force acting on a bead along a polymer chain, then, ideally, $F=ma$ where $m$ is the mass of the bead and $a$ is the acceleration. By (\ref{rousemodel}), $a$ is proportional to the second differential of position.

To be specific, we work with the damped wave equation and a noise that is white in time and colored in space, which will be introduced later. The rigorous definition of white noise is discussed in many references, for example \cite{WalshSPDE86} and \cite{alma9946044013405216}.
The outline of the theorem and some lemmas' proofs are similar to the ones in \cite{MN22}.

We provide an intuitive justification for our main result.

We assume that $l_t(y)$ is constant over $y\in[-R,R]$. Then $l_t(y)=\frac{J}{2R}$. We have
\begin{equation}
	\int_{0}^{T}\int_{-R}^{R}l_t^2(y)dydt=2TR\left(\dfrac{J}{2R}\right)^2=\dfrac{CTJ^2}{R}\label{A1}
\end{equation}
where $C$ is a constant independent from $T$ and $J$. 

We want to find an approximate probability for the colored noise $\dot{F}$.
That is 
\[
\exp\left(-\frac{1}{2}\int_{0}^{T}\int_{0}^{J}\left(\dot{F}(t,x)\right)^2dxdt\right).
\]

 From (\ref{main system}), we substitute for $\dot{F}$. We get an approximate probability of 
\begin{equation*}
	\exp\left(-\frac{1}{2}\int_{0}^{T}\int_{0}^{J}\left[\partial_{t}^{2}u(t,x)+\partial_{t}u(t,x)-\Delta u(t,x)\right]^2dxdt\right).
\end{equation*}

We have that the minimizer $u$ is often constant in $t$, giving us
\begin{equation*}
	\exp\left(-\frac{T}{2}\int_{0}^{J}\left[\Delta u(t,x)\right]^2dx\right).
\end{equation*}

We might think that the minimizer $u$ has a constant value of $|\Delta u|$. Considering the Neumann boundary conditions, a function could be
\[
u(x)=
\begin{cases}
	ax^2-aJ^2/4\quad & x\in[0,J/2]\\
	a(J-x)^2-aJ^2/4\quad & x\in[J/2,J].
\end{cases}
\]

Taking $u(0)=R$ and $u(J)=-R$, we get $a=4R/J^2$ and $|u''(x)|=2a=8R/J^2$. Then 
\begin{equation}
	\frac{T}{2}\int_{0}^{J}\left[\Delta u(t,x)\right]^2dx=CTJ\left(\dfrac{R}{J^2}\right)^2=\dfrac{CTR^2}{J^3}.\label{A2}
\end{equation}

Equating (\ref{A1}) and (\ref{A2}), we get 
\begin{equation*}
	R=CJ^{5/3}.
\end{equation*}

\textbf{Acknowledgements}

The author is very grateful to Carl Mueller for his valuable insights and advice. This work was supported by University of Rochester, Mathematics department.

\section{Preliminary}

Let $(\Omega, \mathcal{F}, (\mathcal{F}_t)_{t\geq0}, P)$ be a probaility space where $\mathcal{F}_t$ is the filtration generated by the white noise in time $(W(s))_{s\leq t}$. In other word, we have $\mathcal{F}_t=\sigma\{W(s):s\leq t\}$.\\

Let  $\mathbb{N}=\{0, 1, 2,\dots\}$ and 
$\big(u(t,x)\big)_{t\geq0, x\in\mathbb{R}}$ be a solution of the following wave equation with the colored noise satisfying initial conditions and the Neumann Boundary condition:

\begin{equation}
	\begin{split}
		\partial_{t}^{2}u(t,x)+\partial_{t}u(t,x)&=\Delta u(t,x)+\dot{F}(t,x)\\
		u(0,x)=u_{0}(x), &\quad \partial_{t}u(0,x)=u_{1}(x) \quad (x, t)\in[0,J]\times\mathbb{R}_+\\
		\partial_{x}u(t,0)&=\partial_{x}u(t,J)=0.
	\end{split}\label{main system}
\end{equation}

Heuristically, the Fourier series of noise $F$ is
\begin{equation*}
	\dot{F}(t,x)=\sum_{n\in\mathbb{N}}\gamma_n\dot{W}_n(t)\varphi_n(x)
\end{equation*}
where 
\begin{enumerate}
	\item[$\cdot$] $(W_n)_{n\in\mathbb{N}}$ are independent and identical distributed white noise in time,
	
	\item[$\cdot$]$(\varphi_n) _{n\in\mathbb{Z}}$ is a complete set of eigenfunctions of the Laplacian $\Delta$ satisfying Neumann boundary condition:
	\begin{align*}
		\varphi_n(x)&=c_n\cos\left(\dfrac{n\pi}{J}x\right), n\in\mathbb{Z},\\
		\Delta \varphi_n&=\lambda_n\varphi_n.
	\end{align*}
	with 
	\[
	c_n=
	\begin{cases}
		\sqrt{\frac{2}{J}}, & n\neq0\\
		\sqrt{\frac{1}{J}}, &n=0
	\end{cases}
	\quad\text{and}\quad
	\lambda_n=
	\begin{cases}
		-\left(\frac{n\pi}{J}\right)^2, &n\neq0\\
		0, &n=0.
	\end{cases}
	\]

	\item[$\cdot$]$(\gamma_n)_{n\in\mathbb{N}}$ is a collection of real numbers such that $(\gamma_n^2)_{n\in\mathbb{N}}$ is a decreasing sequence satisfying the following conditions, 
	
	\begin{equation*}
		\sum_{n\in\mathbb{N}}\gamma_n^2<+\infty \quad \text{ and }\quad
		\gamma^2_n\leq \dfrac{c}{n^\alpha}, \forall n\neq0.
	\end{equation*}
	where $c$ and $\alpha$ are positive constants inpendent from $n$.
	
\end{enumerate}

By Theorem 1.1 of  \cite{Nu20},  the fundamental solution of the damped wave equation on $\mathbb{R}$ is 
\begin{equation*}
	G_t^\mathbb{R}(x)=\frac{1}{2}e^{-t/2}\text{sgn}(t)\text{ }I_0\bigg(\frac{1}{2}\sqrt{t^2-x^2}\bigg)\chi_{[-|t|, |t|]}(x),
\end{equation*}
where $I_0$ is the modified Bessel function of the first kind and with parameter 0.\\

If we consider the Neumann Boundary condition, the fundamental solution of the system (\ref{main system}) is
\begin{equation*}
	G_t(x,y)=\sum_{n\in\mathbb{Z}}G_t^\mathbb{R}(x+y-2nJ)+G_t^\mathbb{R}(x-y-(2n+1)J),\quad x, y\in[0,J].
\end{equation*}

Then the mild solution of (\ref{main system}) is
\begin{equation*}
	\begin{split}
		u(t,x)=
		&\int_{0}^{J}\partial_{t}G_t(x,y)u_0(y)dy+\int_{0}^{J}G_t(x,y)\big(\frac{1}{2}u_0(y)+u_1(y)\big)dy\\
		&+\int_{0}^{t}\int_{0}^{J}G_{t-s}(x,y)F(dyds).
	\end{split}
\end{equation*}

According to Theorem 1.3 of \cite{Nu20}, this mild solution is the unique solution in $\mathcal{C}^1(\mathbb{R}, \mathcal{D}'(\mathbb{R}))$.

By Theorem 5.3 of \cite{Nu20} and Young's inequality, we can show that the Fourier series of $u$ converges in $L^2\left([0,J]\times\Omega\right)$. That is
\begin{equation*}
	u(t,x)=\sum_{n\in\mathbb{N}}a_n(t)\varphi_n(x)
\end{equation*}
where 
\begin{equation*}
	a_n(t)=\begin{cases}
		\frac{\gamma_{n}}{w_n}\int_{0}^{t}e^{\frac{-1+w_n}{2}(t-s)}-e^{\frac{-1-w_n}{2}(t-s)}dW_n(s) & n<\frac{J}{2\pi} \\
		\gamma_{n}\int_{0}^{t}e^{-\frac{1}{2}(t-s)}(t-s)dW_n(s) & n=\frac{J}{2\pi} \\
		\frac{\gamma_{n}}{\omega_n}\int_{0}^{t}e^{-\frac{1}{2}(t-s)}\sin(\omega_n(t-s))dW_n(s) & n>\frac{J}{2\pi}
	\end{cases}
\end{equation*}
where $\omega_n=\sqrt{|1-\left(\frac{2n\pi}{J}\right)^2|}$.

Details of getting the expression of $a_n$ are in Appendix  \ref{dampedwaveequationwiththenoiseF}.

Let $m(\cdot)$ be the Lebesgue measure on $\mathbb{R}$. Then we define an occupation measure and a local time as follows,
\begin{equation*}
	\begin{split}
		L_t(A)&=m\{x\in[0,J]: u(t,x)\in A\}\\
		l_t(y)&=\dfrac{L_t(dy)}{dy}.
	\end{split}
\end{equation*}

If $P_{T,J}$ is the original probability measure of $\left(u(t,x)\right)_{t\in[0,T],x\in[0,J]}$, we define the probability measure $Q_{T,J,\beta}$ as follows:

\begin{equation*}
	Q_{T,J,\beta}(A)=\dfrac{1}{Z_{T,J,\beta}}E[\mathcal{E}_{T,J,\beta}\mathbbm{1}_A].
\end{equation*}

Let $E^{P_{T,J}}$ and $E^{Q_{T,J,\beta}}$ be the expectation with respect to $P_{T,J}$ and $Q_{T,J,\beta}$ respectively. We write $E$ for $E^{P_{T,J}}$. Let
\begin{equation*}
	\begin{split}
		\mathcal{E}_{T,J,\beta}&=\exp\left(-\beta\int_{0}^{T}\int_{-\infty}^{\infty}l_t(y)^2dydt\right),\\
		Z_{T,J,\beta}&=E[\mathcal{E}_{T,J,\beta}]=E^{P_{T,J}}[\mathcal{E}_{T,J,\beta}].
	\end{split}
\end{equation*}
where $\beta$ is a positive parameter.

For ease of notation, we will write
\begin{equation*}
	P_T=P_{T,J},\quad P_T=P_{T,J,\beta},\quad \mathcal{E}_T=\mathcal{E}_{T,J,\beta},\quad Z_T=Z_{T,J,\beta}.
\end{equation*}

We define the radius of $\left(u(t,x)\right)_{t\in[0,T],x\in[0J]}$ to be
\begin{equation}
	R(T,J)=\left[\dfrac{1}{TJ}\int_{0}^{T}\int_{0}^{J}\left(u(t,x)-\bar{u}(t)\right)^2dxdt\right]^{1/2} \label{radius_R}
\end{equation}
where $\bar{u}(t)=\frac{1}{J}\int_{0}^{J}u(t,x)dx$.

\begin{thm}\label{th:1}
	\begin{enumerate}
		\item[(1).] When $0<J\leq1$, for all $\beta>0$, there are constants $\epsilon_1$ and $K_1$ not depending on $\beta$ and $J$ such that 
		\begin{equation}
			\underset{T\rightarrow\infty}{\lim}Q_T\left[\epsilon_0J^{5/3}\leq R(T,J)\leq K_1J^{5/3}\right]=1.
		\end{equation}
		\item[(2).] When $J>1$, if $\beta\cdot J^{1/3}\gg1$ and $\beta\geq J^{25/3}$, there are constants $\epsilon_2$ and $K_2$ not depending on $\beta$ and $J$ such that 
		\begin{equation}
			\underset{T\rightarrow\infty}{\lim}Q_T\left[\epsilon_2J^{5/3}\leq R(T,J)\leq K_2J^{5/3}\right]=1.
		\end{equation}
	\end{enumerate}
\end{thm}

\begin{proof}[\textbf{Outline of proof of Theorem }\ref{th:1}]
	We define 
	\begin{equation*}
		\begin{split}
			A^{(1)}_{T,J}:=\left\{R(T,J)<\epsilon_0J^{5/3}\right\} \quad \text{ and } \quad A^{(2)}_{T,J}:=\left\{R(T,J)>KJ^{5/3}\right\} .
		\end{split}
	\end{equation*}
	
	It suffices to show that for $i=1, 2$, 
	\begin{equation*}
		\underset{T\rightarrow\infty}{\lim}Q_T\left(A^{(i)}_{T,J}\right)=0.
	\end{equation*}
\end{proof}

\section{Lower bound } \label{lowerbound}

For $0<J\leq1$, we will show that $Q_T\left(A^{(1)}_{T,J}\right)$ approaches $0$ as $T$ goes to infinity. First,  we need to find a lower bound of $Z_T$.

\subsection{Stationary solution}

We define a measure $\hat{P}_T^{a}$ that adds a drift depending on $x$ to the colored noise. We  add a drift $a\varphi_1(\cdot)$ to the noise $\dot{F}$, where $a$ is a nonzero constant. Recall that $\varphi_1$ is one of the eigenfunctions of the Laplacian operator.\\
Fixing $T>0$, by Theorem 5.1 from \cite{Da78}, we have
\begin{equation*}
	\begin{split}
		\frac{d\hat{P}_T^{a}}{dP_t}=
		&\exp\bigg(\int_{0}^{T}\int_{0}^{J}a\varphi_1(x)F(dxdt)\\
		&-\frac{1}{2}\int_{0}^{T}\int_{0}^{J}\int_{0}^{J}a^2\varphi_1(x)\varphi_1(y)f(x,y)dxdydt\bigg)
	\end{split}
\end{equation*}
where 
$f(x,y)=\sum_{n\in\mathbb{N}}\gamma_n^2\varphi_n(x)\varphi_n(y)$. 

Let $\hat{\E}$ be the expectation with respect to $\hat{P}_T^{(a)}$. By Jensen's inequality
\begin{equation}
	\begin{split}
		\log Z_T 
		&= \log\hat{\E}\bigg[\exp\bigg(-\beta\int_{0}^{T}\int_{-\infty}^{\infty}\ell_t(y)^2dydt-\log\frac{d\hat{P}_T^{a}}{dP_t}\bigg)\bigg]\\
		&\geq-\beta\hat{\E}\bigg[\int_{0}^{T}\int_{-\infty}^{\infty}\ell_t(y)^2dydt\bigg]-\hat{\E}\bigg[\log\frac{d\hat{P}_T^{a}}{dP_t}\bigg].\label{logZT}
	\end{split}
\end{equation}

Recall the Fourier series of $u$ is
\begin{equation*}
	u(t,x)=\sum_{n\in\mathbb{N}}a_n(t)\varphi_{n}(x).
\end{equation*}

By Fubini's theorem, it is not hard to find that
\begin{equation}
	\begin{split}
		\bar{u}(t)
		&=a_0(t)\varphi_0(x).\label{ubaroft}
	\end{split}
\end{equation}

Then we have 
\begin{equation*}
	u(t,x)-\bar{u}(t)=\sum_{n\neq 0}a_n(t)\varphi_{n}(x).
\end{equation*}

Let 
\begin{equation*}
	U(t,x)=u(t,x)-\bar{u}(t).
\end{equation*}

For $n\neq 0$, we consider the process $a_n$ in the future time. We define $\tilde{a}_n$ and $\tilde{U}$ as the following:
\begin{equation}
	\tilde{a}_n(t)=\begin{cases}
		\frac{\gamma_{n}}{w_n}\int_{-\infty}^{t}e^{\frac{-1+w_n}{2}(t-s)}-e^{\frac{-1-w_n}{2}(t-s)}dW_n(s) & J>2\pi n\\
		\gamma_{n}\int_{-\infty}^{t}e^{-\frac{1}{2}(t-s)}(t-s)dW_n(s) & J=2\pi n\\
		\frac{\gamma_{n}}{\omega_n}\int_{-\infty}^{t}e^{-\frac{1}{2}(t-s)}\sin(\omega_n(t-s))dW_n(s) & 0<J<2\pi n.
	\end{cases}\label{a_tilde}
\end{equation}
and $\tilde{U}$ has the same relation with $\tilde{a}_n$ as $U$ with $a_n$. That is
\begin{equation}
	\tilde{U}(t,x)=\sum_{n\in\mathbb{N}_+}\tilde{a}_n(t)\varphi_n(x). \label{U_tilde}
\end{equation}

After taking the drift, let $\hat{a}_n$ and $\hat{U}$ be the expressions corresonding to $\tilde{a}$ and $\tilde{U}$ respectively. So we have 

When $n\neq1$, we have $\hat{a}_n=\tilde{a}_n$.

When $n=1$, 
\begin{equation*}
	\hat{a}_1(t)=\begin{cases}
		\frac{\gamma_{1}}{\omega_1}\int_{-\infty}^{t}e^{\frac{-1+\omega_1}{2}(t-s)}-e^{\frac{-1-\omega_1}{2}(t-s)}\big(dW_1(s)+\dfrac{a}{\gamma_1}ds\big) & J>2\pi \\
		\gamma_{1}\int_{-\infty}^{t}e^{-\frac{1}{2}(t-s)}(t-s)\big(dW_1(s)+\dfrac{a}{\gamma_1}ds\big) & J=2\pi \\
		\frac{\gamma_{1}}{\omega_1}\int_{-\infty}^{t}e^{-\frac{1}{2}(t-s)}\sin(\omega_1(t-s))\big(dW_1(s)+\dfrac{a}{\gamma_1}ds\big) & 0<J<2\pi ,
	\end{cases}
\end{equation*}
and

\begin{equation*}
	\hat{U}(t,x)=\sum_{n\in\mathbb{N}_+}\hat{a}(t)\varphi_n(x).
\end{equation*}

Heuristically, 
\begin{equation*}
	\hat{F}(t.x)(t,x)+a\varphi_1(x)=\gamma_1\left(\dot{W_1}(t)+\dfrac{a}{\gamma_1}\right)\varphi_1(x)+\sum_{n\neq1}\gamma_n\dot{W_n}(t)\varphi_n(x).
\end{equation*}

It is easy to check that for each $n\in\mathbb{N}_+$, $\hat{a}_n$ is weakly stationary.

Since $\{\hat{a}_n\}_{n\in\mathbb{N}_+}$ is jointly Gaussian, $\{\hat{a}_n\}_{n\in\mathbb{N}_+}$ is strong stationary.  Then $\hat{U}(t,\cdot)$ is stationary in time $t$.

Let $g_{t, x_1, x_2}$ be the density function of $\tilde{U}(t, x_1)-\tilde{U}(t, x_2)$ under $P$, and  $\hat{g}$ be the density function of $\hat{U}(t, x_1)-\hat{U}(t, x_2)$ under $\hat{P}^a$.

Since $\hat{U}(t,\cdot)$ is stationary in $t$ and
\begin{equation*}
	\begin{split}
		\hat{U}(t,x_1)-\hat{U}(t,x_2)
		&=\sum_{n\in\mathbb{N_+}}\hat{a}_n(t)\big(\varphi_n(x_1)-\varphi_n(x_2)\big).
	\end{split}
\end{equation*}

We have 
\begin{equation*}
	\hat{g}_{t,x_1, x_2}=\hat{g}_{0,x_1, x_2}, \quad\forall t\in\mathbb{R}.
\end{equation*}
To simplify our notation, we write $\hat{g}_{x_1, x_2}$ instead of $\hat{g}_{0, x_1, x_2}$.

By Lemma 2.5 of \cite{MN22}, we have
\begin{equation*}
	\begin{split}
		\hat{\E}\bigg[\int_{-\infty}^{\infty}l_t(y)^2dy\bigg]
		&=\int_{0}^{J}\int_{0}^{J}\hat{g}_{t, x_1, x_2}(0)dx_2dx_1\\
		&=\int_{0}^{J}\int_{0}^{J}\hat{g}_{ x_1, x_2}(0)dx_2dx_1\\
		&=2\int_{0}^{J}\int_{x_1}^{J}\hat{g}_{x_1, x_2}(0)dx_2dx_1\\
		&=2\int_{0}^{J}\int_{x_1}^{J}g_{x_1, x_2}\bigg(D(x_1,x_2)\bigg)dx_2dx_1\\
		&=C_1\int_{0}^{J}\int_{x_1}^{J}\frac{1}{\sigma(x_1, x_2)}\exp\bigg(-\frac{D(x_1, x_2)^2}{2\sigma(x_1, x_2)^2}\bigg)dx_2dx_1.
	\end{split}
\end{equation*}
where $C_1$ is a constant,  $\sigma$ is the standard deviation of $\tilde{U}(0,x_1)-\tilde{U}(0,x_2)$ and $D(x_1,x_2)$ is the drift term and $D(x_1, x_2)=d(t)\left(\varphi_1(x_1)-\varphi_1(x_2)\right)$. In Appendix \ref{driftterm}, we showed that 
\begin{equation*}
	d(t)=\begin{cases}
		\dfrac{a}{\omega_1}\left(\dfrac{2}{1-\omega_1}-\dfrac{2}{1+\omega_1}\right) & J>2\pi \\
		2a & J=2\pi \\
		\dfrac{4a}{5\omega_1} & 0<J<2\pi .
	\end{cases}
\end{equation*}

Since $\exp\bigg(-\frac{D(x_1, x_2)^2}{2\sigma(x_1, x_2)^2}\bigg)\leq 1$ for all $0\leq x_1\leq x_2\leq J$, in order to get an upper bound of $\hat{\E}\bigg[\int_{-\infty}^{\infty}l_t(y)^2dy\bigg]$, it is enough to find a lower bound of $\sigma$.

\begin{lemma}
	$\sigma(x_1,x_2)^2$ is bounded below by $\tilde{C}\frac{\left|x_1-x_2\right|}{J^2}$ for some constant $\tilde{C}$.
\end{lemma}\label{lowerboundlemma3.1}
Proof of Lemma \ref{lowerboundlemma3.1}  is in Appendix \ref{lowerboundofsigmax1x2square}.

Now we go back to $\hat{\E}\bigg[\int_{-\infty}^{\infty}l_t(y)^2dy\bigg]$. 

\subsubsection{for $0<J\leq 1$}

\begin{equation*}
	\begin{split}
		\hat{\E}\bigg[\int_{-\infty}^{\infty}l_t(y)^2dy\bigg]
		&=2\int_{0}^{J}\int_{x_1}^{J}g_{x_1, x_2}\bigg(D(x_1, x_2)\bigg)dx_2dx_1\\
		&\leq C_1\int_{0}^{J}\int_{x_1}^{J}\frac{1}{\sigma(x_1, x_2)}\exp\bigg(-\frac{D(x_1, x_2)^2}{\sigma(x_1, x_2)^2}\bigg)dx_2dx_1\\
		\leq C_1\mathcal{I}_1+C_1\mathcal{I}_2
	\end{split}
\end{equation*}
where 
\begin{equation*}
	\mathcal{I}_1=\int_0^J\int_{x_1}^J\frac{1}{\sigma(x_1, x_2)}\exp\bigg(-\frac{D(x_1, x_2)^2}{\sigma(x_1, x_2)^2}\bigg)\mathbbm{1}_{\{x_1+x_2\leq J\}}dx_2dx_1
\end{equation*}
and 
\begin{equation*}
	\mathcal{I}_2=\int_0^J\int_{x_1}^J\frac{1}{\sigma(x_1, x_2)}\exp\bigg(-\frac{D(x_1, x_2)^2}{\sigma(x_1, x_2)^2}\bigg)\mathbbm{1}_{\{x_1+x_2\geq J\}}dx_2dx_1.
\end{equation*}

Then 
\begin{equation*}
	\begin{split}
		\mathcal{I}_1
		&\leq C_1\int_0^J\int_{x_1}^J\frac{J}{\sqrt{\left|x_2-x_1\right|}}\mathbbm{1}_{\{x_1+x_2\leq J\}}dx_2dx_1\\
		&\leq C_1J\int_0^{J}\int_{0}^{J-x_1}\frac{1}{\sqrt{p}}dpdx_1\\
		&=\dfrac{4}{3}C_1J^{\frac{5}{2}}\\
		&=C_2J^{\frac{5}{2}}.
	\end{split}
\end{equation*}
where $C_2=\frac{4}{3}C_1$. In the second inequality, we take $p=x_2-x_1$. 

Now we work with $\mathcal{I}_2$.

Since in this case, $x_2\geq\min\{x_1, J-x_1\}$, we have the following
\begin{equation*}
	\begin{split}
		\mathcal{I}_2
		\leq \int_0^J\int_{x_1}^J\frac{J}{\sqrt{\left|x_2-x_1\right|}}dx_2dx_1.
	\end{split}
\end{equation*}

Let $p=x_2-x_1$, then we have
\begin{equation*}
	\begin{split}
		\int_0^J\int_{x_1}^J\frac{J}{\sqrt{\left|x_2-x_1\right|}}dx_2dx_1
		&\leq C_1J\int_0^{J}\int_{0}^{J-x_1}\frac{1}{\sqrt{p}}dpdx_1\\
		&=C_2J^\frac{5}{2}.
	\end{split}
\end{equation*}

We have 
\begin{equation*}
	\hat{\E}\bigg[\int_{-\infty}^{\infty}l_t(y)^2dy\bigg]\leq C_3J^\frac{5}{2}
\end{equation*}
where $C_3=2C_2$ is a constant.

\newpage
\subsubsection{for $J>1$}

Starting on page 7

\begin{equation*}
	(\ast)
	:=
	\hat{E}\left[\int_{-\infty}^{+\infty}l_t(y)^2dy\right]
	=
	C_1\int_{0}^{J}\int_{x_1}^{J}
	\frac{1}{\sigma(x_1, x_2)}
	\exp\left(-\frac{D(x_1, x_2)^2}{2\sigma^2(x_1, x_2)}\right)
	dx_2dx_1,
\end{equation*}
where $C_1=\frac{1}{\sqrt{2\pi}}$.

By Lemma 3.1, let $C_2=C_1\cdot\tilde{C}^{-\frac{1}{2}}$, we have
\begin{equation*}
	(\ast)
	\leq
	C_2\int_{0}^{J}\int_{x_1}^{J}
	\frac{J}{\sqrt{|x_2-x_1|}}
	\exp\left(-\frac{D(x_1, x_2)^2J^2}{32\tilde{C}}\right)
	dx_2dx_1.
\end{equation*}

Since $d(t)=C'a$ where$C'=\begin{cases}
\frac{1}{\omega_1}\left(\frac{2}{1-\omega_1}-\frac{2}{1+\omega_1}\right)
&\quad J>2\pi;\\
2
&\quad J=2\pi;\\
\frac{4}{5\omega_1}
&\quad 0<J<2\pi
\end{cases}$ from (\ref{C'driftconstant}), 
let $C_3=\frac{(C')^2}{32\tilde{C}}$, then
\begin{equation*}
	(\ast)
	\leq
	C_2J\int_{0}^{J}\int_{x_1}^{J}
	\frac{J}{\sqrt{|x_2-x_1|}}
	\exp\left(-C_3a^2J^2
	\left(\varphi_1(x_1)-\varphi_1(x_2)\right)^2
	\right)
	dx_2dx_1.
\end{equation*}

Recall $\varphi_1(x)=\sqrt{\frac{2}{J}}\cos\left(\frac{\pi}{J}x\right)$, let $C_4=2C_3$, 
\begin{equation*}
	(\ast)
	\leq
	C_2J\int_{0}^{J}\int_{x_1}^{J}
	\frac{J}{\sqrt{|x_2-x_1|}}
	\exp\left(-C_3a^2J
	\left(\cos\left(\frac{\pi}{J}x_1\right)-\cos\left(\frac{\pi}{J}x_2\right)\right)^2
	\right)
	dx_2dx_1.
\end{equation*}

By the proof of Lemma 3.1, there is a constant $C_5$ such that
\begin{equation*}
	\left(\cos\left(\frac{\pi}{J}x_1\right)-\cos\left(\frac{\pi}{J}x_2\right)\right)^2
	\geq C_5\frac{|x_1-x_2|^2}{J^2}.
\end{equation*}

Let $C_6=C_4\cdot C_5$, we get
\begin{equation*}
	(\ast)
	\leq
	C_2J\int_{0}^{J}\int_{x_1}^{J}
	\frac{J}{\sqrt{|x_2-x_1|}}
	\exp\left(-C_6\frac{a^2}{J}(x_2-x_1)^2\right)
	dx_2dx_1.
\end{equation*}

Let $p=x_2-x_1$, then 
\begin{equation*}
	(\ast)
	\leq
	C_2J\int_{0}^{J}\int_{0}^{J}
	\frac{1}{\sqrt{p}}\exp\left(-C_6\frac{a^2}{J}p^2\right)dpdx_1
\end{equation*}

Let $A=C_6\frac{a^2}{J}$, then by change of variable, we define $y=Ap^2$, and get
$p=\left(\frac{y}{A}\right)^{\frac{1}{2}}$ and $dp=A^{-\frac{1}{2}}\cdot\frac{1}{2}y^{-\frac{1}{2}}dy$.
Then
\begin{align*}
	(\ast)
	&\leq
	C_2J\int_{0}^{J}\int_{0}^{AJ^2}
	\frac{A^{1/4}}{y^{1/4}}\exp(-y)A^{-\frac{1}{2}}\cdot\frac{1}{2}y^{-\frac{1}{2}}dydx_1.
\end{align*}

Let $C_7=\frac{1}{2}C_2$, then we have
\[
(\ast)\leq
C_7J^2A^{-\frac{1}{4}}\gamma\left(\frac{1}{4}, AJ^2\right),
\]
where $\gamma(s,z)$ is the incomplete gamma function. 

An expansion of $\gamma(s,z)$ is 
\[
\gamma(s,z)=\frac{z^s}{s}\mathcal{M}(s,s+1,-z)
\]
where $\mathcal{M}(a,b,z)$ is the Kummer's confluent hypergeometric function. Then

\[
(\ast)
\leq
C_7J^2A^{-\frac{1}{4}}
\left(4\left(AJ^2\right)^{1/4}\mathcal{M}\left(\frac{1}{4},\frac{5}{4},-AJ^2\right)\right).
\]

For large $|z|$, we have
\[
\mathcal{M}(a,b,z)\sim
\Gamma(b)\left(\frac{e^z\cdot z^{a-b}}{\Gamma(a)}+\frac{(-z)^{(-a)}}{\Gamma(b-a)}\right).
\]

For $\left\lvert AJ^2\right\rvert$ large enough, 
\[
\mathcal{M}\left(\frac{1}{4},\frac{5}{4},-AJ^2\right)\sim
\Gamma\left(\frac{5}{4}\right)
\left(
\frac{e^{-AJ^2}\left(-AJ^2\right)^{-1}}{\Gamma\left(\frac{1}{4}\right)}
+
\frac{\left(AJ^2\right)^{-\frac{1}{4}}}{\Gamma(1)}
\right)
\]

and 

\[
(\ast)
\leq
\frac{4C_7J^2\Gamma\left(\frac{5}{4}\right)}{\Gamma(1)}A^{-\frac{1}{4}}
=
C_8\frac{J^{9/4}}{a^{1/2}}
\]\label{Jgeq1}
where $C_8=\frac{4C_6C_7\Gamma(\frac{5}{4})}{\Gamma(1)}.$

Since 
\begin{equation*}
	\hat{E}\bigg[\int_0^T\int_{-\infty}^{\infty}l_t(y)^2dydt\bigg]=\hat{E}\bigg[\int_0^T\int_{-\infty}^{\infty}l_0(y)^2dydt\bigg]=T\hat{E}\bigg[\int_{-\infty}^{\infty}l_0(y)^2dy\bigg],
\end{equation*}
back to inequality (\ref{logZT}), we have
\begin{equation*}
	\log Z_T\geq 
	-\beta T\hat{E}\bigg[\int_{-\infty}^{\infty}l_0(y)^2dy\bigg]
	-\hat{E}\bigg[\log\frac{d\hat{P}_T^{a}}{dP_T}\bigg].
\end{equation*}

Now we work on $\hat{E}\bigg[\log\frac{d\hat{P}_T^{a}}{dP_T}\bigg]$.

By Theorem 5.1 of \cite{Da78}, we have
\begin{equation*}
	\frac{d\hat{P}_T^{a}}{dP_T}=\exp\bigg(\int_0^T\int_0^Ja\varphi_1(x)F(dxdt)-\frac{1}{2}\int_0^T\int_0^J\int_0^Ja^2\varphi_1(x)\varphi_1(y)f(x,y)dxdydt\bigg)
\end{equation*}
where $f(x,y)=\sum_{n=0}^{\infty}\gamma_n^2\varphi_n(x)\varphi_n(y)$.

Then
\begin{equation*}
	\hat{\E}\bigg[\log\frac{d\hat{P}_T^{a}}{dP_T} \bigg]
	=\hat{\E}\bigg[\int_0^T\int_0^Ja\varphi_1(x)F(dxdt)\bigg]
	-\frac{1}{2}\int_0^T\int_0^J\int_0^Ja^2\varphi_1(x)\varphi_1(y)f(x-y)dxdydt.
\end{equation*}

Let 
\begin{equation*}
	\begin{split}
		\hat{\zeta}(T,a)&=\exp\bigg(\frac{1}{2}\int_0^T\int_0^J\int_0^Ja^2\varphi_1(x)\varphi_1(y)f(x-y)dxdydt\bigg),\\
		F(T, a\varphi_1)&=\int_0^T\int_0^Ja\varphi_1(x)F(dxdt).
	\end{split}
\end{equation*}

Then 
\begin{equation}
	\begin{split}
		\hat{\E}\bigg[\log\frac{d\hat{P}_T^{a}}{dP_T} \bigg]
		&=\hat{\E}\big[F(T, a\varphi_1)\big]-\log\hat{\zeta}(T,a)\\
		&=\E\big[F(T,a\varphi_1)\frac{d\hat{P}_T^{a}}{dP_T}\big]-\log\hat{\zeta}(T,a)\\
		&=\frac{a}{\hat{\zeta}(T,a)}E\big[F(T,\varphi_1)\exp(aF(T,\varphi_1))\big]-\log\hat{\zeta}(T,a)\\
		&=\frac{a}{\hat{\zeta}(T,a)}\frac{d}{da}\E\big[\exp(aF(T,\varphi_1))\big]-\log\hat{\zeta}(T,a).\label{exp(aF)}
	\end{split}
\end{equation}

Now, let $X=F(T,\varphi_1)$ and $\psi(a)=\E[\exp(aX)]$. Then $X$ is normally distributed with mean zero and variance $\sigma_X^2$, where
\begin{equation*}
	\sigma_X^2=\int_0^T\int_0^J\int_0^J\varphi_1(x)\varphi_1(y)f(x-y)dxdydt=\frac{2}{a^2}\log\hat{\zeta}(T,a).
\end{equation*}

Then 
\begin{equation*}
	\begin{split}
		\psi(a)&=\exp\bigg(\frac{a^2}{2}\sigma_X^2\bigg),\\
		\frac{d}{da}\psi(a)&=\sigma_X^2a\exp(\frac{a^2}{2}\sigma_X^2).
	\end{split}
\end{equation*}

We bring these two expressions to (\ref{exp(aF)}), we get
\begin{equation*}
	\begin{split}
		\hat{\E}\big[F(T, a\varphi_1)\big]
		&=\frac{a}{\hat{\zeta}(T,a)}\sigma_X^2a\exp\bigg(\frac{a^2}{2}\sigma_X^2\bigg)\\
		&=\frac{a^2\exp\bigg(\frac{a^2}{2}\sigma_X^2\bigg)}{\hat{\zeta}(T,a)}\cdot\frac{2}{a^2}\log\hat{\zeta}(T,a)\\
		&=2\exp\bigg(\frac{a^2}{2}\sigma_X^2\bigg)\frac{\log\hat{\zeta}(T,a)}{\hat{\zeta}(T,a)},
	\end{split}
\end{equation*}
and 
\begin{equation*}
	\begin{split}
		\hat{\E}\bigg[\log\frac{d\hat{P}_T^{a}}{dP_T} \bigg]
		&=\hat{\E}\big[F(T, a\varphi_1)\big]-\log\hat{\zeta}(T,a)\\
		&=2\exp\bigg(\frac{a^2}{2}\sigma_X^2\bigg)\frac{\log\hat{\zeta}(T,a)}{\hat{\zeta}(T,a)}-\log\hat{\zeta}(T,a)\\
		&=2\exp\bigg(\log\hat{\zeta}(T,a)\bigg)\frac{\log\hat{\zeta}(T,a)}{\hat{\zeta}(T,a)}-\log\hat{\zeta}(T,a)\\
		&=\log\hat{\zeta}(T,a).
	\end{split}
\end{equation*}

Since
\begin{equation*}
	\begin{split}
		\log\hat{\zeta}(T,a)
		&=\frac{1}{2}\int_0^T\int_0^J\int_0^Ja^2\varphi_1(x)\varphi_1(y)f(x,y)dxdydt\\
		&=\frac{a^2T}{2}\int_0^J\int_0^J\varphi_1(x)\varphi_1(y)f(x,y)dxdy\\
		&=\frac{a^2TJ\gamma_1^2}{4},
	\end{split}
\end{equation*}

we have
\begin{equation*}
	\hat{\E}\bigg[\log\frac{d\hat{P}_T^a}{dP_T}\bigg]=\frac{a^2TJ\gamma_1^2}{4}.
\end{equation*}

Due to the inequality (\ref{logZT}), we get
\begin{equation}
	\underset{T\rightarrow\infty}{\liminf}\frac{1}{T}\log Z_T
	\geq
	-C_3\beta J^\frac{5}{2}-\frac{a^2J\gamma_1^2}{4}.\label{1/TlogZT}
\end{equation}

\subsection{Estimate}

Recall
\begin{equation*}
	\bar{u}(t)=\frac{1}{J}\int_{0}^{J}u(t,x)dx.
\end{equation*}

In	\cite{MN22}, $\theta_u$ and $R_\varphi$ are defined as the following
\begin{equation*}
	\begin{split}
		\theta_u(t,J)&:=\bigg[\frac{1}{J}\int_{0}^{J}\big(u(t,x)-\bar{u}(t)\big)^2dx\bigg]^{1/2}, \quad 0\leq t\leq T,\\
		R_\varphi(T,J)&=\bigg(\frac{1}{T}\int_{0}^{T}\theta_\varphi(t,J)^2dt\bigg)^{1/2}.
	\end{split}
\end{equation*}

We define the event
\begin{equation*}
	A=A_{K,T,J}=\{R(T,J)\leq K\}.
\end{equation*}

\begin{lemma}
	On the set $A$, we have
	\begin{equation*}
		\left|\{t\in[0,T]:\theta_u^2(t,J)\leq2K^2\}\right|\geq\frac{T}{2}.
	\end{equation*}
\end{lemma}

\begin{proof}
	We prove this by contradiction.\\
	Suppose on $A$,
	\begin{equation*}
		\left|\{t\in[0,T]:\theta_u^2(t,J)>2K^2\}\right|\geq\frac{T}{2}.
	\end{equation*}
	
	Then 
	\begin{equation*}
		\begin{split}
			\int_{0}^{T}\theta_u^2(t,J)dt
			&>\int_{0}^{T}\theta_u^2(t,J)\mathbbm{1}_{\{\theta_u^2(t,J)>2K^2\}}dt\\
			&>2K^2\cdot\frac{T}{2}\\
			&=K^2T.
		\end{split}
	\end{equation*}
	
	But 
	\begin{equation*}
		R(T,J)\leq K
	\end{equation*}
	is equivalent to 
	\begin{equation*}
		\frac{1}{T}\int_{0}^{T}\theta_u^2(t,J)dt\leq K^2,
	\end{equation*}
	and it is equivalent to
	\begin{equation*}
		\int_{0}^{T}\theta_u^2(t,J)dt\leq K^2T.
	\end{equation*}
	We have the contradiction.
\end{proof}

\begin{lemma}
	If $\theta_u(t,J)^2\leq 2K^2$, we have 
	\begin{equation*}
		\left|\{x\in[0, J]: u(t,x)\in[\bar{u}(t)-2K, \bar{u}(t)+2K]\}\right|>\frac{J}{2}.
	\end{equation*}
\end{lemma}

\begin{proof}
	We prove this by contradiction.
	Suppose when $\theta_u(t,J)^2\leq 2K^2$, then
	\begin{equation*}
		\left|\{x\in[0, J]: u(t,x)\in[\bar{u}(t)-2K, \bar{u}(t)+2K]\}\right|\leq\frac{J}{2}.
	\end{equation*}
	We also have 
	\begin{equation*}
		\left|\{x\in[0, J]: \left|u(t,x)-\bar{u}(t)\right|>2K\}\right|>\frac{J}{2}.
	\end{equation*}
	
	Then we get
	\begin{equation*}
		\int_{0}^{J}\big(u(t,x)-\bar{u}(t)\big)^2dx>(2K)^2\cdot\frac{J}{2}=2K^2J.
	\end{equation*}
	
	But by definition of $\theta_u$, if $\theta_u^2(t,J)\leq 2K^2$, we have 
	\begin{equation*}
		\frac{1}{J}\int_{0}^{J}\big(u(t,x)-\bar{u}(t)\big)^2dx\leq2K^2.
	\end{equation*}
	We have
	\begin{equation*}
		\int_{0}^{J}\big(u(t,x)-\bar{u}(t)\big)^2dx\leq2K^2J.
	\end{equation*}
	This is a contradiction.
\end{proof}

Let $d_t^{\pm}=\bar{u}(t)\pm2K$, then $d_t^+-d_t^-=4K$. Recall
\begin{equation*}
	\begin{split}
		L_t(A)&=m\{x\in[0,J]:u(t,x)\in A\},\\
		l_t^u(y)&=L_t(dy)/dy.
	\end{split}
\end{equation*}

Then
\begin{equation*}
	\begin{split}
		\bigg|&\left\{t\in[0,T]:\int_{d_k^-}^{d_k^+}l_t^u(x)dx\geq\frac{J}{2}\right\}\bigg|
		=\left|\left\{t\in[0,T]:L_t\big((d_k^-,d_k^+)\big)\geq\frac{J}{2}\right\}\right|\\
		&=\left|\left\{t\in[0,T]:\left|\{x\in[0, J]: u(t,x)\in[\bar{u}(t)-2K, \bar{u}(t)+2K]\}\right|>\frac{J}{2}\right\}\right| \\
		&\geq\frac{T}{2}.
	\end{split}
\end{equation*}

Then we get
\begin{equation*}
	\begin{split}
		\int_{0}^{T}\int_{-\infty}^{+\infty}l_t^u(y)^2dydt
		&\geq 4K\int_{-\infty}^{+\infty}\bigg(\int_{d_k^-}^{d_k^+}l_t^u(y)^2\frac{dy}{4K}\bigg)dt\\
		&\geq 4K\int_{-\infty}^{+\infty}\bigg(\int_{d_k^-}^{d_k^+}l_t^u(y)\frac{dy}{4K}\bigg)^2dt\\
		&\geq \frac{TJ^2}{32K}.
	\end{split}
\end{equation*}
The third inequality is due to Jensen's inequality.

We let $K=\epsilon_0J^{5/3}$ where $C$ is a positive constant, then 
\begin{equation*}
	\int_{0}^{T}\int_{-\infty}^{+\infty}l_t^u(y)^2dydt\geq \frac{TJ^2}{32\epsilon_0J^{5/3}}.
\end{equation*}

Recall the definitions
\begin{equation*}
	\begin{split}
		\mathcal{E}_{T, J, \beta}&=\exp\bigg(-\beta\int_{0}^{T}\int_{-\infty}^{+\infty}l_t^u(y)^2dydt\bigg),\\
		Z_{T, J, \beta}&=\E\big[\mathcal{E}_{T,J,\beta}\big],\\
		Q_{T,J,\beta}(A)&=\frac{1}{Z_{T, J, \beta}}\E\big[\mathcal{E}_{T,J,\beta}\mathbbm{1}_A\big].
	\end{split}
\end{equation*}

Then we get

\begin{equation*}
	\begin{split}
		\E\big[\mathcal{E}_{T,J,\beta}&\mathbbm{1}_{\{R(T,J)<\epsilon_0J^{5/3}\}}\big]\\
		&=\exp\bigg(-\beta\int_{0}^{T}\int_{-\infty}^{+\infty}l_t^u(y)^2\mathbbm{1}_{\{R(T,J)<\epsilon_0J^{5/3}\}}dydt\bigg)\\
		&\leq\exp\bigg(-\beta\frac{TJ^2}{32\epsilon_0J^{5/3}}\bigg).
	\end{split}
\end{equation*}

By (\ref{1/TlogZT}), we have
\begin{equation*}
	\begin{split}
		\underset{T\rightarrow\infty}{\lim}\frac{1}{T}\log Q_T(A_{T,J}^{(1)})
		&\leq\underset{T\rightarrow\infty}{\lim}\frac{1}{T}\log\E\big[\mathcal{E}_{T,J,\beta}\mathbbm{1}_{\{R(T,J)<\epsilon_0J^{5/3}\}}\big]-\underset{T\rightarrow\infty}{\liminf}\frac{1}{T}\log Z_T\\
		&\leq\underset{T\rightarrow\infty}{\lim}\frac{1}{T}\bigg(-\beta\frac{TJ^2}{32\epsilon_0J^{5/3}}\bigg)+C_3\beta J^{\frac{5}{2}}+\frac{a^2J\gamma_1^2}{4}\\
		&=-\beta\frac{J^\frac{1}{3}}{32\epsilon_0}+C_3\beta J^{\frac{5}{2}}+\frac{a^2J\gamma_1^2}{4}.
	\end{split}
\end{equation*}

When $0<J\leq1$, let $a^2=\beta$, then 
\begin{align*}
	\underset{T\rightarrow\infty}{\lim}\frac{1}{T}\log Q_T(A_{T,J}^{(1)})
	&\leq-\beta\frac{J^\frac{1}{3}}{32\epsilon_0}+C_3\beta J^{\frac{5}{2}}+\frac{\gamma_1^2}{4}\beta J.
\end{align*}
We have 
\begin{align*}
	\underset{T\rightarrow\infty}{\lim}\frac{1}{T}\log Q_T(A_{T,J}^{(1)})
	&\leq\left(-\dfrac{1}{32\epsilon_0}+C_3+\dfrac{\gamma_1^2}{4}\right)\beta J^{\frac{1}{3}}.
\end{align*}
Hence by choosing $\epsilon_0$ small enough we get the result. 

When $J\geq1$ and $a^2J$ large enough,  by (\ref{Jgeq1}), we have
\[
\underset{T\rightarrow\infty}{\lim}\frac{1}{T}\log Q_T(A_{T,J}^{(1)})
\leq
-\beta\frac{J^\frac{1}{3}}{32\epsilon_0}
+C_8\beta\frac{J^{9/4}}{a^{1/2}}
+\frac{a^2J\gamma_1^2}{4}.
\]

By taking $a^2=\beta J^{-2/3}$, when $\beta\geq J^{\frac{25}{3}}$, we have
\begin{align*}
	\underset{T\rightarrow\infty}{\lim}\frac{1}{T}\log Q_T(A_{T,J}^{(1)})
	&\leq
	-\beta\frac{J^\frac{1}{3}}{32\epsilon_0}
	+
	C_8\beta\frac{J^{29/12}}{\beta^{1/4}}
	+
	\frac{\beta J^{1/3}\gamma_1^2}{4}\\
	&\leq
	\left(
	-\frac{1}{32\epsilon_0}+C_8+\frac{\gamma_1^2}{4}
	\right)
	\beta J^{1/3}.
\end{align*}

By choosing $\epsilon_0$ small enough, we get the result.

\section{Upper bound}

In this section, we will  show $	\underset{T\rightarrow\infty}{\lim}Q_T\left(A^{(i)}_{T,J}\right)=0$.
Since $\mathcal{E}_{T,J,\beta}\leq 1$ and
we have found a lower bound of $\log Z_T$ in section \ref{lowerbound}, it is enough to prove that for $K:=CJ^{5/3}$ where $C$ is a constant independent from $J$, we have $\underset{T\rightarrow+\infty}{\lim}P\left(R(T,J)\geq K\right)^{1/T}=0.$

Recall from (\ref{a_tilde}) and (\ref{U_tilde}),
\begin{equation*}
	\tilde{a}_n(t)=\begin{cases}
		\frac{\gamma_{n}}{w_n}\int_{-\infty}^{t}e^{\frac{-1+w_n}{2}(t-s)}-e^{\frac{-1-w_n}{2}(t-s)}dW_n(s) & J>2\pi n \\
		\gamma_{n}\int_{-\infty}^{t}e^{-\frac{1}{2}(t-s)}(t-s)dW_n(s) & J=2\pi n \\
		\frac{\gamma_{n}}{\omega_n}\int_{-\infty}^{t}e^{-\frac{1}{2}(t-s)}\sin(\omega_n(t-s))dW_n(s) & 0<J<2\pi n.
	\end{cases}
\end{equation*}
and $\tilde{U}(t,x)=\sum_{n\neq 0}\tilde{a}_n(t)\varphi_n(x)$, we define
\begin{equation*}
	S_T^{n}=\int_{0}^{T}\big(\tilde{a}_n(t)\big)^2dt.
\end{equation*}
Then 
\begin{equation*}
	\begin{split}
		R^2(T,J)
		&=\sum_{n=1}^{\infty}\frac{1}{TJ}S_T^{n}.
	\end{split}
\end{equation*}

\begin{prop}
	Let $K=CJ^{5/3}$ where $C$ is a constant independent from $J$, then
	\begin{equation*}
		\underset{T\rightarrow+\infty}{\lim}P\left(R(T,J)\geq K\right)^{1/T}=0.
	\end{equation*}
\end{prop}

\begin{proof}
	
	Let $c_0=\sum_{n=1}^{+\infty}\frac{1}{n^2}=\frac{6}{\pi^2}$. We have
	\begin{equation*}
		\begin{split}
			P\left(R(T,J)\geq K\right)
			&=P\left(R(T,J)^2\geq K^2\right)\\
			&\leq \sum_{n=1}^{\infty}P\left(\dfrac{1}{TJ}S_T^n>\frac{1}{c_0}K^2n^{-2}\right).
		\end{split}
	\end{equation*}

	We will show the work with three cases: (1) $J<2\pi n$; (2) $J=2\pi n$; and (3) $J>2\pi n$. The proofs from each case are similar. We will omit some details in case (2) and (3).
	
	\textbf{Case (1) When $J<2\pi n$} , 
	
	\textbf{Step 1.} For each $n>\frac{J}{2\pi}$, we will find an upper bound of $P\left(\dfrac{1}{TJ}S_T^n>Kn^{-2}\right).$
	
	Let $A_n(t)=\int_{-\infty}^{t}e^{-\frac{1}{2}(t-s)}\sin(\omega_n(t-s))dW_n(s)$. Then
	\begin{equation*}
		\tilde{a}_n(t)=\frac{\gamma_n}{\omega_n}A_n(t)
	\end{equation*}
	and 
	\begin{equation*}
		\dfrac{1}{TJ}S_T^n=\dfrac{1}{TJ}\bigg(\dfrac{\gamma_n}{\omega_n}\bigg)^2\int_{0}^{T}A_n^2(t)dt.
	\end{equation*}
	
	We have
	\begin{equation*}
		\begin{split}
			A_n(t)
			&=\int_{-\infty}^{t}e^{-\frac{1}{2}(t-s)}\sin(\omega_n(t-s))dW_n(s)\\
			&=e^{-\frac{1}{2}t}\sin(\omega_nt)\int_{-\infty}^{t}e^{\frac{1}{2}s}\cos(\omega_ns)dW_n(s)
			-e^{-\frac{1}{2}t}\cos(\omega_nt)\int_{-\infty}^{t}e^{\frac{1}{2}s}\sin(\omega_ns)dW_n(s).
		\end{split}
	\end{equation*}
	
	Starting with the first integral, $\int_{-\infty}^{t}e^{\frac{1}{2}s}\cos(\omega_ns)dW_n(s)$ is a time-changed two-sided Brownian motion. Since
	\begin{equation*}
		E\big[\big(\int_{-\infty}^{t}e^{\frac{1}{2}s}\cos(\omega_ns)dW_n(s)\big)^2\big]
		=
		\int_{-\infty}^{t}e^{s}\cos^2(\omega_ns)ds
		\leq 
		e^t,
	\end{equation*}
	when $t\geq 0$, we have 
	\begin{equation*}
		\int_{-\infty}^{t}e^{\frac{1}{2}s}\cos(\omega_ns)dW_n(s)\overset{D}{=}B_{\int_{-\infty}^{t}e^{s}\cos^2(\omega_ns)ds}\leq \underset{0\leq r\leq t}{\sup}B_{e^r}
	\end{equation*}
	where $\{B_t\}$ is a standard Brownian motion. 
	
	Similarly, we get
	\begin{equation*}
		\int_{-\infty}^{t}e^{\frac{1}{2}s}\sin(\omega_ns)dW_n(s)\overset{D}{=}B_{\int_{-\infty}^{t}e^{s}\sin^2(\omega_ns)ds}\leq \underset{0\leq r\leq t}{\sup}B_{e^r}.
	\end{equation*}
	
	To ease notation, we define
	\begin{equation*}
		\begin{split}
			B^c_t&:=\int_{-\infty}^{t}e^{\frac{1}{2}s}\cos(\omega_ns)dW_n(s)\\
			B^s_t&:=\int_{-\infty}^{t}e^{\frac{1}{2}s}\sin(\omega_ns)dW_n(s)\\
			\tilde{B}_{T}&:=\underset{0\leq t\leq T}{\sup}|B_t|.
		\end{split}
	\end{equation*}
	
	Let $\lambda$ be a real number, 
	\begin{equation*}
		\begin{split}
			P\bigg(\big(A_n(t)\big)^2>\lambda\bigg)
			&=P\bigg(\bigg(e^{-\frac{1}{2}t}\sin(\omega_nt)\int_{-\infty}^{t}e^{\frac{1}{2}s}\cos(s)dW_n(s)\\
			&\qquad-e^{-\frac{1}{2}t}\cos(\omega_nt)\int_{-\infty}^{t}e^{\frac{1}{2}s}\sin(s)dW_n(s)\bigg)^2>\lambda\bigg)\\
			&\leq P\bigg(2e^{-t}(B^c_t)^2+2e^{-t}(B^s_t)^2>\lambda\bigg)\\
			&\leq P\bigg(16e^{-t}\tilde{B}_{e^t}^2>\lambda\bigg).
		\end{split}
	\end{equation*}
	
	Given $0\leq s\leq t$, 
	\begin{equation}
		\begin{split}
			\tilde{B}_t
			&=\underset{0\leq r\leq t}{\sup}|B_r|=\underset{0\leq r\leq t}{\sup}|B_r-B_s+B_s|\\
			&\leq \underset{0\leq r\leq t}{\sup}|B_r-B_s|+|B_s|\\
			&\leq \underset{ 0\leq r\leq s}{\sup}|B_r-B_s|+\underset{s\leq r\leq t}{\sup}|B_r-B_s|+|B_s|\\
			&\leq 2\underset{0\leq r\leq s}{\sup}|B_r|+\underset{s\leq r\leq t}{\sup}|B_r-B_s|+|B_s|\\
			&\leq 3\underset{0\leq r\leq s}{\sup}|B_r|+\underset{s\leq r\leq t}{\sup}|B_r-B_s|.\label{ineqoftildeB}
		\end{split}
	\end{equation}
	
	Let
	\begin{equation*}
		\tilde{B}_{s,t}:=\underset{s\leq r\leq t}{\sup}|B_r-B_s|.
	\end{equation*}
	
	Note that $\tilde{B}_{s,t}$ is independent from $\mathcal{F}_s$.  In conclusion, we have the following relations
	\begin{equation}
		\tilde{B}_t\leq \tilde{B}_{s,t}+3\tilde{B}_s, \quad \text{and }\quad \tilde{B}_{s,t}\overset{D}{=}\tilde{B}_{t-s}.\label{Btildeupperbound}
	\end{equation}
	
	Let $\tau >0$. By Markov's inequality,
	\begin{equation}
		\begin{split}
			P\bigg(\frac{1}{TJ}S_n>Kn^{-2}\bigg)
			&= P\bigg(\int_{0}^{T}A_n^2(t)dt>\frac{K}{n^2}\big(\frac{\omega_n}{\gamma_n}\big)^2TJ\bigg)\\
			&=P\bigg(e^{\tau\int_{0}^{T}A_n^2(t)dt}>e^{\tau\frac{K}{n^2}(\frac{\omega_n}{\gamma_n})^2TJ}\bigg)\\
			&\leq E\big[e^{\tau\int_{0}^{T}A_n^2(t)dt}\big]\cdot e^{-\tau\frac{K}{n^2}(\frac{\omega_n}{\gamma_n})^2TJ}\\
			&\leq E\big[e^{16\tau\int_{0}^{T}e^{-t}\tilde{B}_{e^t}^2dt}\big]\cdot e^{-\tau\frac{K}{n^2}(\frac{\omega_n}{\gamma_n})^2TJ}.\label{tauMI}
		\end{split} 
	\end{equation}
	
	\textbf{Step 2.} We will find an upper bound of $\int_{0}^{T}e^{-t}\tilde{B}_{e^t}^2dt$.
	
	\begin{lemma}
		Let $T>2$ be a positive integer and  $a,b>0$. Then 
		\begin{align*}
			\int_{0}^{T}e^{-at}\tilde{B}^2_{e^{bt}}dt
			&=
			\int_{0}^{1}e^{-at}\tilde{B}^2_{e^{bt}}dt+\tilde{B}_{e^{b}}^2
			\sum_{k=1}^{T-1}18^k\int_{k}^{k+1}e^{-at}dt\\
			&\quad +\sum_{m=1}^{T-2}
			\bigg[
			\bigg(
			\sum_{k=m+1}^{T-1}2\cdot18^{k-m}\int_{k}^{k+1}
			e^{-at}\tilde{B}^2_{e^{bm},e^{b(m+1)}}dt
			\bigg)\\
			&\hspace{4.5em}
			+2\int_{m}^{m+1}e^{-at}\tilde{B}^2_{e^{bm},e^{bt}}dt
			\bigg].
		\end{align*}
	\end{lemma}\label{upperboundlemma1}
	
	The proof of Lemma 4.1 is in Appendix \ref{proofofintegraleexpbrownianmotion}.
	
	Now, we take $a=b=1$. Then Lemma 4.1 gives
	\begin{equation}
		\begin{split}
			\int_{0}^{T}e^{-t}\tilde{B}^2_{e^{t}}dt
			&=
			\int_{0}^{1}e^{-t}\tilde{B}^2_{e^{t}}dt+
			\tilde{B}_{e}^2\sum_{k=1}^{T-1}18^k\int_{k}^{k+1}e^{-t}dt\\
			&\quad 
			+\sum_{m=1}^{T-2}
			\bigg[\bigg(
			\sum_{k=m+1}^{T-1}2\cdot18^{k-m}\int_{k}^{k+1}
			e^{-t}\tilde{B}^2_{e^{m},e^{(m+1)}}dt
			\bigg)\\
			&\hspace{4.5em}
			+2\int_{m}^{m+1}e^{-t}\tilde{B}^2_{e^{m},e^{t}}dt
			\bigg].
		\end{split}\label{lemma4.1case1}
	\end{equation}
	
	Let 
	\begin{equation*}
		\begin{split}
			(\text{I})&:=\int_{0}^{1}e^{-t}\tilde{B}^2_{e^{t}}dt
			+\tilde{B}_{e}^2\sum_{k=1}^{T-1}18^k\int_{k}^{k+1}e^{-t}dt\\
			(\text{II})&:=\sum_{m=1}^{T-2}\bigg[
			\bigg(\sum_{k=m+1}^{T-1}
			2\cdot 18^{(k-m)}\int_{k}^{k+1}
			e^{-t}\tilde{B}^2_{e^{m},e^{(m+1)}}dt\bigg)\\
			&\qquad\qquad+2\int_{m}^{m+1}e^{-t}\tilde{B}^2_{e^{m},e^{t}}dt
			\bigg].
		\end{split}
	\end{equation*}

	Recall that  $\mathcal{F}_t$ is the filtration generated by $\{B_s | \hspace{0.05cm} s\leq t\}$.
	
	Clearly, $(\text{I})\in\mathcal{F}_e$. And in (I), since $\frac{18}{e}<7$
	\begin{equation*}
			\sum_{k=1}^{T-1}18^k\int_{k}^{k+1}e^{-t}dt\\
			\leq \sum_{k=1}^{T-1}\big(\frac{18}{e}\big)^k\\
			\leq\frac{7^T-7}{6}.
	\end{equation*}
	
	Then we have 
	\begin{equation*}
		(\text{I})\leq \int_{0}^{1}e^{-t}\tilde{B}^2_{e^t}dt+\bigg(\frac{7^T-7}{6}\bigg)\tilde{B}^2_e\leq \int_{0}^{1}e^{-t}\tilde{B}^2_{e^t}dt+7^T\tilde{B}^2_e.
	\end{equation*}
	Since $\tilde{B}_{e^t}\leq \tilde{B}_e, \forall t\in[0,1]$, 
	\begin{equation*}
		(\text{I})\leq \tilde{B}^2_e(1+7^T).
	\end{equation*}
	
	Now we work with (II).
	
	For each $m$ in $\{1,\dots, T-2\}$,
	\begin{equation*}
		\bigg[\bigg(\sum_{k=m+1}^{T-1}
		2\cdot18^{(k-m)}\int_{k}^{k+1}
		e^{-t}
		\tilde{B}^2_{e^m,e^{m+1}}dt\bigg)+2\int_{m}^{m+1}e^{-t}\tilde{B}^2_{e^m,e^t}dt\bigg]\in\mathcal{F}_{e^{m+1}}
	\end{equation*} 
	and 
	\begin{equation*}
		\bigg[\bigg(\sum_{k=m+1}^{T-1}
		2\cdot18^{(k-m)}\int_{k}^{k+1}
		e^{-t}
		\tilde{B}^2_{e^m,e^{m+1}}dt\bigg)+2\int_{m}^{m+1}e^{-t}\tilde{B}^2_{e^m,e^t}dt\bigg]\indep\mathcal{F}_{e^{m}}.
	\end{equation*} 
	
	Also we have
	\begin{equation*}
		\begin{split}
			\sum_{k=m+1}^{T-1}&
			2\cdot18^{(k-m)}\int_{k}^{k+1}e^{-t}\tilde{B}^2_{e^m,e^{m+1}}dt\\
			&=\tilde{B}^2_{e^m,e^{m+1}}\sum_{k=m+1}^{T-1}2\cdot18^{(k-m)}
			\int_{k}^{k+1}e^{-t}dt\\
			&=\tilde{B}^2_{e^m,e^{m+1}}\cdot2\cdot18^{-m}\sum_{k=m+1}^{T-1}18^k\int_{k}^{k+1}e^{-t}dt\\
			&\leq \tilde{B}^2_{e^m,e^{m+1}}\cdot2\cdot18^{-m}\sum_{k=m+1}^{T-1}(\frac{18}{e})^k\\
			&\leq \tilde{B}^2_{e^m,e^{m+1}}\cdot2\cdot18^{-m}\cdot\frac{7^T-7^{m+1}}{6}\\
			&\leq \tilde{B}^2_{e^m,e^{m+1}}\cdot2\cdot18^{-m}\cdot 7^T.
		\end{split}
	\end{equation*}
	
	Let $\alpha_m=2\cdot18^{-m}\cdot 7^T$. 
	
	Recall $\tilde{B}_{s,t}=\underset{s\leq r\leq t}{\sup}|B_r-B_s|$. Then if $t\leq p$, we have $\tilde{B}_{s,t}\leq\tilde{B}_{s,p}$.
	
	Therefore,
	\begin{equation*}
		\begin{split}
			(\text{II})
			&\leq\sum_{m=1}^{T-2}\bigg[\tilde{B}^2_{e^m,e^{m+1}}\cdot\alpha_m+2\int_{m}^{m+1}e^{-t}\tilde{B}^2_{e^m,e^t}dt\bigg]\\
			&\leq\sum_{m=1}^{T-2}\tilde{B}^2_{e^m,e^{m+1}}\big(\alpha_m+2e^{-m}\big).
		\end{split}
	\end{equation*}
	
	Back to (\ref{lemma4.1case1}), we get
	\begin{equation*}
		\begin{split}
			\int_{0}^{T}e^{-t}\tilde{B}^2_{e^t}dt 
			&\leq \tilde{B}^2_e(1+7^T)+\sum_{m=1}^{T-2}\tilde{B}^2_{e^m,e^{m+1}}\big(\alpha_m+2e^{-m}\big).
		\end{split}
	\end{equation*}
	
	\textbf{Step 3.} We will choose a proper value for $\tau$ and find an explicit upper bound of\\ $E\big[e^{16\tau\int_{0}^{T}e^{-t}\tilde{B}_{e^t}^2dt}\big]$.\\

	Back to (\ref{tauMI}), we have 
	\begin{equation}
		\begin{split}
			E\big[&e^{16\tau\int_{0}^{T}e^{-t}\tilde{B}_{e^t}^2dt}\big]
			\leq E\big[e^{16\tau\left(\tilde{B}^2_e(1+7^T)+\sum_{m=1}^{T-2}\tilde{B}^2_{e^m,e^{m+1}}\big(\alpha_m+2e^{-m}\big)\right)}\big]\\
			&=E\big[e^{16\tau\tilde{B}^2_e(1+7^T)}\big]
			\cdot \prod_{m=1}^{T-2}E\big[e^{16\tau\tilde{B}^2_{e^m,e^{m+1}}\big(\alpha_m+2e^{-m}\big)}\big].\label{mterms}
		\end{split}
	\end{equation}
	
	Since $e^{16\tau\tilde{B}^2_e(1+7^T)}$ is a nonnegative random variable, we have
	\begin{equation*}
		\begin{split}
			E\big[e^{16\tau(1+7^T)\tilde{B}_e^2}\big]
			&=\int_{1}^{+\infty}P\big(e^{16\tau(1+7^T)\tilde{B}_e^2}\geq x\big)dx\\
			&=\int_{1}^{+\infty}4P\bigg(B_e\geq \sqrt{\frac{\log x}{16\tau(1+7^T)}}\bigg)dx.
		\end{split}
	\end{equation*}
	
	\begin{lemma}
		Let $X\sim \mathcal{N}(0, \sigma^2)$ and $\gamma>0$ with $1-2\gamma\sigma^2>0$, then 
		\begin{equation}
			\int_{1}^{+\infty}P\left(X\geq\sqrt{\dfrac{\log x}{\gamma}}\right)dx\leq\dfrac{\sigma^2\gamma}{\sqrt{1-2\gamma\sigma^2}}.
		\end{equation}
	\end{lemma}
	
	The proof of the lemma is in Appendix  \ref{integraloftailprobability}.
	
	Let $\gamma=16\tau(1+7^T)$ with $1-2\gamma e>0$. We choose $\tau\in\left(0,\dfrac{1}{32e(1+7^T)}\right)$. Since $B_e\sim\mathcal{N}(0, e)$, by Lemma 4.2, we have 
	\begin{equation*}
		E\big[e^{16\tau(1+7^T)\tilde{B}_e^2}\big]
		\leq4\cdot\dfrac{e\gamma}{\sqrt{1-2\gamma e}}=\dfrac{64e\tau(1+7^T)}{\sqrt{1-32e\tau(1+7^T)}}.
	\end{equation*}

	Now we work on the other factors in (\ref{mterms}). For $m\in\{1, 2, \dots, T-3, T-2\}$, we have
	
	\begin{equation*}
		\begin{split}
			E\Bigg[\exp\bigg(16\tau\tilde{B}^2_{e^m,e^{m+1}}\big(\alpha_m+2e^{-m}\big)\bigg)\Bigg]
			&=\int_{1}^{+\infty}P\big(\exp\big(16\tau\tilde{B}^2_{e^m,e^{m+1}}(\alpha_m+2e^{-m})\big)\geq x\big)dx\\
			&=\int_{1}^{+\infty}4P\bigg(B_{e^{m+1}-e^m}\geq \sqrt{\frac{\log x}{16\tau(\alpha_m+2e^{-m})}}\bigg)dx.
		\end{split}
	\end{equation*}
	
	Let $\gamma_m=16\tau(\alpha_m+2e^{-m})$ and $\sigma_m=e^{m+1}-e^m$.
	
	When $1-2\sigma_m^2\gamma_m>0$, we choose $\tau\in\left(0,\dfrac{1}{32\sigma_m^2(\alpha_m+2e^{-m})}\right)$. By Lemma 4.2, we have
	\begin{equation*}
		\begin{split}
			E\big[\exp\bigg(16\tau\tilde{B}^2_{e^m,e^{m+1}}\big(\alpha_m+2e^{-m}\big)\bigg)\big]
			&\leq 4\cdot \dfrac{\sigma_m^2\gamma_m}{\sqrt{1-2\gamma_m\sigma_m^2}}\\
			&=\dfrac{64\tau(e^{m+1}-e^m)(\alpha_m+2e^{-m})}{\sqrt{1-32\tau(e^{m+1}-e^m)(\alpha_m+2e^{-m})}}.
		\end{split}
	\end{equation*}

	Let $\tau=\frac{1}{32e^T\cdot8^T}$. Then $\tau<\frac{1}{32(1+7^T)e}$ and $\tau<\frac{1}{32\sigma_m^2(\alpha_m+2e^{-m})}$ for all $m$. According to (\ref{mterms}),
	\begin{equation*}
		\begin{split}
			E\big[&e^{16\tau\int_{0}^{T}e^{-t}\tilde{B}_{e^t}^2dt}\big]
			\leq
			\dfrac{64e\tau(1+7^T)}{\sqrt{1-32e\tau(1+7^T)}}
			\cdot\prod_{m=1}^{T-2}
			\dfrac{64\sigma_m^2\tau(\alpha_m+2e^{-m})}{\sqrt{1-32\sigma_m^2\tau(\alpha_m+2e^{-m})}}\\
			&=\dfrac{\frac{2(1+7^T)}{e^{T-1}8^T}}{\sqrt{1-\frac{(1+7^T)}{e^{T-1}8^T}}}\cdot\prod_{m=1}^{T-2}\dfrac{\frac{2\alpha_m\sigma_m^2}{e^T8^T}}{\sqrt{1-\frac{\alpha_m\sigma_m^2}{e^T8^T}}}\\
			&=\dfrac{2(1+7^T)}{\sqrt{e^{T-1}8^T}\sqrt{e^{T-1}8^T-(1+7^T)}}\\
			&\quad\cdot\prod_{m=1}^{T-2}\dfrac{2(2\cdot18^{-m}\cdot 7^T+2e^{-m})(e^{m+1}-e^m)}{\sqrt{e^{T}8^T}\sqrt{e^{T}8^T-(2\cdot18^{-m}\cdot 7^T+2e^{-m})(e^{m+1}-e^m)}}.
		\end{split}
	\end{equation*}
	
	Since $1+7^T<8^T$ and $\forall m\in\{1,\dots, T-2\}$, $(2^{1-m}\cdot9^{-m}\cdot 7^T+2e^{-m})<8^T$, when $T>2$, we have
	\begin{equation*}
		\begin{split}
			E\big[e^{16\tau\int_{0}^{T}e^{-t}\tilde{B}_{e^t}^2dt}\big]
			&\leq\dfrac{2(1+7^T)}{\sqrt{e^{T-1}8^T}\sqrt{e^{T-1}8^T-8^T}}\\
			&\quad\cdot\prod_{m=1}^{T-2}\dfrac{2(2\cdot18^{-m}\cdot 7^T+2e^{-m})e^{m+1}}{\sqrt{e^{T}8^T}\sqrt{e^{T}8^T-8^Te^{T-1}}}\\
			&=\dfrac{2\cdot2^{T-2}(1+7^T)}{8^T\sqrt{e^{2T-2}-e^{T-1}}}\cdot\dfrac{1}{8^{T(T-2)}(\sqrt{e^{2T}-e^{2T-1}})^{(T-2)}}\\
			&\quad\cdot\prod_{m=1}^{T-2}\bigg[\big(2\cdot18^{-m}\cdot e^{m+1}\big)\cdot 7^T+2e\bigg].
		\end{split}
	\end{equation*}
	
	Since $2\cdot18^{-m}\cdot e^{m+1}\leq 1$ and $2e\leq7^{T}$ when $T\geq3$, we have
	\begin{equation*}
		\prod_{m=1}^{T-2}\bigg[\big(2\cdot18^{-m}\cdot e^{m+1}\big)\cdot 7^T+2e\bigg]
		\leq2^T\cdot7^{T(T-2)}=8^{T/3}\cdot7^{T(T-2)}.
	\end{equation*}
	
	Then we have 
	\begin{equation*}
		\begin{split}
			E\big[e^{16\tau\int_{0}^{T}e^{-t}\tilde{B}_{e^t}^2dt}\big]
			&\leq \dfrac{2\cdot2^{T-2}(1+7^T)8^{T/3}\cdot7^{T(T-2)}}{8^{T^2-T}\cdot\sqrt{e^{2T-2}-e^{T-1}}\cdot e^{T^2-2T}\cdot(\sqrt{1-e^{-1}})^{T-2}}\\
			&\leq \dfrac{2(1+7^T)\cdot7^{T(T-2)}}{8^{T^2-\frac{5}{3}T}\cdot\sqrt{e^{2T-2}-e^{T-1}}\cdot e^{T^2-2T}\cdot(\sqrt{1-e^{-1}})^{T-2}}.
		\end{split}
	\end{equation*}
	
	
	\textbf{Step 4.} We are ready to find upper bounds of $P\bigg(\frac{1}{TJ}S_n>Kn^{-2}\bigg)$ for each $n>\frac{J}{2\pi}$ and show that  $\underset{T\rightarrow+\infty}{\lim}\sum_{n\geq N}P\bigg(\frac{1}{TJ}S_n>Kn^{-2}\bigg)^{1/T}=0$.\\

	Going back to (\ref{tauMI}), 
	\begin{equation*}
		\begin{split}
			P\bigg(&\frac{1}{TJ}S_n>Kn^{-2}\bigg)
			\leq E\big[e^{16\tau\int_{0}^{T}e^{-t}\tilde{B}_{e^T}^2dt}\big]\cdot e^{-\tau\frac{K}{n^2}(\frac{\omega_n}{\gamma_n})^2TJ}\\
			&\leq \dfrac{2(1+7^T)\cdot7^{T^2-2T}}{8^{T^2-\frac{5}{3}T}\cdot\sqrt{e^{2T-2}-e^{T-1}}\cdot e^{T^2-2T}\cdot(\sqrt{1-e^{-1}})^{T-2}}\cdot e^{-\tau\frac{K}{\gamma_n^2}(\frac{\omega_n}{n})^2TJ}
		\end{split} 
	\end{equation*}
	where $\tau=\dfrac{1}{32e^T8^T}$.\\
	
	We denote $\ceil{x}$ as the ceiling of all real number $x$.\\
	
	Since $\omega_n=\dfrac{\sqrt{4k_n^2-1}}{2}$ and $k_n=\dfrac{n\pi}{J}$, we consider the following two cases
	\begin{enumerate}
	\item[(1).] 	if $\ceil[\big]{\dfrac{J}{2\pi}}>\dfrac{J}{2\pi}$, for all $n>\dfrac{J}{2\pi}$ and $n\in\mathbb{Z}$:
	\begin{equation*}
		\dfrac{\omega_n^2}{n^2}=\dfrac{4k_n^2-1}{4n^2}=\dfrac{\frac{4n^2\pi^2}{J^2}-1}{4n^2}=\dfrac{4n^2\pi^2-J^2}{4n^2J^2}=\dfrac{\pi^2}{J^2}-\dfrac{1}{4n^2}\geq\dfrac{\pi^2}{J^2}-\dfrac{1}{(\ceil{J/\pi})^2}>0.
	\end{equation*}

	\item[(2).] If $\ceil[\big]{\dfrac{J}{2\pi}}=\dfrac{J}{2\pi}$, the minimum value that $n$ can take is $\ceil[\big]{\dfrac{J}{2\pi}}+1$:
	\begin{equation*}
		\dfrac{\omega_n^2}{n^2}=\dfrac{4k_n^2-1}{4n^2}=\dfrac{\frac{4n^2\pi^2}{J^2}-1}{4n^2}=\dfrac{4n^2\pi^2-J^2}{4n^2J^2}=\dfrac{\pi^2}{J^2}-\dfrac{1}{4n^2}\geq\dfrac{\pi^2}{J^2}-\dfrac{1}{(\ceil{J/\pi}+1)^2}>0.
	\end{equation*}
	\end{enumerate}
	
	We define $\epsilon(J)=
	\begin{cases}
	\frac{\pi^2}{J^2}-\frac{1}{(\ceil{J/\pi})^2}, 
	\quad &\ceil[\big]{\frac{J}{2\pi}}>\frac{J}{2\pi};\\
	\frac{\pi^2}{J^2}-\frac{1}{4n^2}\geq\frac{\pi^2}{J^2}-\frac{1}{(\ceil{J/\pi}+1)^2},
	\quad & \ceil[\big]{\frac{J}{2\pi}}=\frac{J}{2\pi}
	\end{cases}
	$,

	then
	\begin{align*}
		P\bigg(\frac{1}{TJ}S_n>Kn^{-2}\bigg)
		&\leq \dfrac{2(1+7^T)7^{T^2-2T}}{8^{T^2-\frac{5}{3}T}\cdot\sqrt{e^{2T-2}-e^{T-1}}\cdot e^{T^2-2T}\cdot(\sqrt{1-e^{-1}})^{T-2}}\\
		&\quad\cdot e^{-\tau K\epsilon(J)(\frac{1}{\gamma_n})^2TJ}.
	\end{align*}
	
	Since $\gamma_{n}^2\leq c'n^{-\alpha}$ for some positive constant $c'$ and $\alpha$,  $-\dfrac{1}{\gamma_{n}^2}\leq -cn^{\alpha}$ where $c c'=1$. 
	
	Let $N=
	\begin{cases}
	\frac{J}{2\pi}+1, 
	&\quad\text{ if }\ceil[\big]{\frac{J}{2\pi}}=\frac{J}{2\pi};\\
	\ceil[\big]{\frac{J}{2\pi}}, 
	&\quad \text{ if }\ceil[\big]{\frac{J}{2\pi}}>\frac{J}{2\pi}
	\end{cases}$,
	then 
	\begin{equation*}
		\begin{split}
			\sum_{n\geq N}&P\bigg(\frac{1}{TJ}S_n>Kn^{-2}\bigg)
			\leq \dfrac{2(1+7^T)7^{T^2-2T}}{8^{T^2-\frac{5}{3}T}\cdot\sqrt{e^{2T-2}-e^{T-1}}\cdot e^{T^2-2T}\cdot(\sqrt{1-e^{-1}})^{T-2}}\\
			&\qquad\qquad\qquad\qquad\cdot\bigg(\exp\bigg(-\tau K\epsilon(J)\big(\frac{1}{\gamma_{N}}\big)^2TJ\bigg)+\int_{N}^{+\infty}\exp(-\tau K\epsilon(J) cx^{\alpha}TJ)dx\bigg)\\
			&=\dfrac{2(1+7^T)7^{T^2-2T}}{8^{T^2-\frac{5}{3}T}\cdot\sqrt{e^{2T-2}-e^{T-1}}\cdot e^{T^2-2T}\cdot(\sqrt{1-e^{-1}})^{T-2}\cdot\exp(\frac{1}{32e^T8^T}K\epsilon(J)(\frac{1}{\gamma_{N}})^2TJ) }\\
			&\quad +\left(\dfrac{2(1+7^T)7^{T^2-2T}}{8^{T^2-\frac{5}{3}T}\cdot\sqrt{e^{2T-2}-e^{T-1}}\cdot e^{T^2-2T}\cdot(\sqrt{1-e^{-1}})^{T-2}}\right)\\
			&\qquad\cdot\bigg(\int_{N}^{+\infty}\exp(-\tau K\epsilon(J) cx^{\alpha}TJ)dx\bigg)\\
			&=:\circled{1}+\circled{2}.
		\end{split}
	\end{equation*}
	
	
	First, we note that
	\begin{equation*}
		\underset{T\rightarrow+\infty}{\lim}\left(\exp\bigg(\frac{1}{32e^T8^T}K\epsilon(J)\bigg(\frac{1}{\gamma_{N}^2}\bigg)TJ\bigg)\right)^{1/T}=1.
	\end{equation*}
	
	Secondly, for $T$ large enough,  $\sqrt{e^{2T-2}-e^{T-1}}\geq\sqrt{e^{T-1}}$, we have
	\begin{equation*}
		\left(\dfrac{1}{\sqrt{e^{2T-2}-e^{T-1}}\cdot(\sqrt{1-e^{-1}})^{T-2}}\right)^{1/T}
		\leq\left(\frac{1}{\sqrt{e^{T-1}}}\right)^{1/T}\cdot\left(\frac{1}{\sqrt{1-e^{-1}}}\right)^{1-\frac{2}{T}}.
	\end{equation*}

	Now, we work on $\circled{1}^{1/T}$.
	
	\begin{align*}
		\underset{T\rightarrow+\infty}{\lim}\circled{1}^{1/T}
	&= \underset{T\rightarrow+\infty}{\lim}\bigg[\dfrac{1}{\sqrt{e^{2T-2}-e^{T-1}}\cdot(\sqrt{1-e^{-1}})^{T-2}\cdot
		\exp\left(\frac{1}{32e^T8^T}K\epsilon\left(\frac{1}{\gamma_{N}}\right)^2TJ\right)}\\
	&\qquad\qquad\cdot\dfrac{2(1+7^T)7^{T^2-2T}}{8^{T^2-\frac{5}{3}T}\cdot e^{T^2-2T}}\bigg]^{1/T}.
	\end{align*}

	We write the right hand side of the equation as a multiplication of three limits.
	
	\begin{align*}
	\underset{T\rightarrow+\infty}{\lim}\circled{1}^{1/T}
	&=\bigg(\underset{T\rightarrow+\infty}{\lim}\hspace{0.1in}\left(\dfrac{1}{\exp(\frac{1}{32e^T8^T}K\epsilon(\frac{1}{\gamma_{N}})^2TJ)}\right)^{1/T}\bigg)\\
	&\quad\cdot \bigg(\underset{T\rightarrow+\infty}{\lim}\hspace{0.1in}\left(\dfrac{1}{\sqrt{e^{2T-2}-e^{T-1}}\cdot(\sqrt{1-e^{-1}})^{T-2}}\right)^{1/T}\bigg)\\
	&\quad\cdot\bigg[\underset{T\rightarrow+\infty}{\lim}\bigg(\dfrac{2\cdot7^{T^2-2T}}{8^{T^2-\frac{5}{3}T}\cdot e^{T^2-2T}}+\dfrac{2\cdot7^{T^2-T}}{8^{T^2-\frac{5}{3}T}\cdot e^{T^2-2T}}\bigg)^{1/T}\bigg].
	\end{align*}

	Then by the previous dicussion, we derive an upper bound of $	\underset{T\rightarrow+\infty}{\lim}\circled{1}^{1/T}$. That is 
	\begin{align*}
		\underset{T\rightarrow+\infty}{\lim}\circled{1}^{1/T}
		&\leq\bigg(\underset{T\rightarrow+\infty}{\lim}\hspace{0.1in}\left(\dfrac{1}{\exp(\frac{1}{32e^T8^T}K\epsilon(\frac{1}{\gamma_{N}})^2J)}\right)\bigg)\\
		&\quad\cdot \bigg(\underset{T\rightarrow+\infty}{\lim}\hspace{0.1in}\left(\frac{1}{\sqrt{e^{T-1}}}\right)^{1/T}\cdot\left(\frac{1}{\sqrt{1-e^{-1}}}\right)^{1-\frac{2}{T}}\bigg)\\
		&\quad\cdot\bigg[\underset{T\rightarrow+\infty}{\lim}\bigg(\dfrac{4\cdot7^{T^2-T}}{8^{T^2-\frac{5}{3}T}\cdot e^{T^2-2T}}\bigg)^{1/T}\bigg].
	\end{align*}

	Since
	\begin{equation}
		\underset{T\rightarrow+\infty}{\lim}\bigg(\dfrac{4\cdot7^{T^2-T}}{8^{T^2-\frac{5}{3}T}\cdot e^{T^2-2T}}\bigg)^{1/T}=0,
	\end{equation}
	we derive that $ \underset{T\rightarrow+\infty}{\lim}\circled{1}^{1/T}=0$.
	
	Next, we consider $\circled{2}$. According to the moments of stretched exponential function, we have the following inequality,
	\begin{equation*}
		\int_{N}^{+\infty}\exp(-\tau TK\epsilon(J) cx^{\alpha}J)dx\leq\tilde{K}\cdot\Gamma\bigg(\dfrac{1}{\alpha}\bigg)(8^Te^T)^{\frac{1}{\alpha}}\left(\dfrac{1}{T}\right)^{1/\alpha}.
	\end{equation*}
	where $\tilde{K}=\dfrac{1}{\alpha}\bigg(\dfrac{32}{K\epsilon(J) cJ}\bigg)^{\frac{1}{\alpha}}$ and $\Gamma$ is the gamma function.
	
	When $T>1$, $T^{1/\alpha}>1$ for all $\alpha>0$, 
	\begin{equation*}
		\int_{N}^{+\infty}\exp(-\tau TK\epsilon(J) cx^{\alpha}J)dx\leq\tilde{K}\cdot\Gamma\bigg(\dfrac{1}{\alpha}\bigg)(8^Te^T)^{\frac{1}{\alpha}}.
	\end{equation*}
	
	Then we have
	\begin{equation*}
		\begin{split}
			\circled{2}
			&=\dfrac{2(1+7^T)7^{T^2-2T}}{8^{T^2-\frac{5}{3}T}\cdot\sqrt{e^{2T-2}-e^{T-1}}\cdot e^{T^2-2T}\cdot(\sqrt{1-e^{-1}})^{T-2}}\cdot\bigg(\int_{N}^{+\infty}\exp(-\tau K\epsilon(J) cx^{\alpha}TJ)dx\bigg)\\
			&\leq\tilde{K}\Gamma\bigg(\dfrac{1}{\alpha}\bigg)\dfrac{2(1+7^T)7^{T^2-2T}8^{\frac{T}{\alpha}}e^{\frac{T}{\alpha}}}{8^{T^2-\frac{5}{3}T}\cdot\sqrt{e^{2T-2}-e^{T-1}}\cdot e^{T^2-2T}\cdot(\sqrt{1-e^{-1}})^{T-2}}\\
			&=:\tilde{K}\Gamma\bigg(\dfrac{1}{\alpha}\bigg)\circled{2a}.
		\end{split}
	\end{equation*}
	
	Since $\tilde{K}$ and $\Gamma\left(\dfrac{1}{\alpha}\right)$ are independent from $T$, $\underset{T\rightarrow+\infty}{\lim}\circled{2}^{1/T}=\underset{T\rightarrow+\infty}{\lim}\circled{2a}^{1/T}$.
	
	Then
	\begin{equation*}
		\begin{split}
			\underset{T\rightarrow+\infty}{\lim}\circled{2a}^{1/T}
			&=\underset{T\rightarrow+\infty}{\lim}\dfrac{2^{\frac{1}{T}}(1+7^T)^{\frac{1}{T}}7^{T-2}\cdot8^{\frac{1}{\alpha}}\cdot e^{\frac{1}{\alpha}}}{8^{T-\frac{5}{3}}\cdot\left(\sqrt{e^{2T-2}-e^{T-1}}\right)^{\frac{1}{T}}\cdot e^{T-2}\cdot(\sqrt{1-e^{-1}})^{1-\frac{2}{T}}}\\
			&\leq \dfrac{e^2\cdot8^{\frac{5}{3}}\cdot8^{\frac{1}{\alpha}}\cdot e^{\frac{1}{\alpha}}}{7^2}
			\left(\underset{T\rightarrow+\infty}{\lim}
			\dfrac{2^{\frac{2}{T}}\cdot\left(7^{T\cdot\frac{1}{T}}\right)\cdot7^{T}}{8^T\cdot\left(\sqrt{e^{T-1}}\right)^{\frac{1}{T}}\cdot e^T\cdot (\sqrt{1-e^{-1}})^{1-\frac{2}{T}}}\right)
		\end{split}
	\end{equation*}
	
	Since
	\begin{align*}
		\underset{T\rightarrow+\infty}{\lim}\dfrac{2^{\frac{2}{T}}\cdot\left(7^{T\cdot\frac{1}{T}}\right)\cdot7^{T}}{8^T\cdot\left(\sqrt{e^{T-1}}\right)^{\frac{1}{T}}\cdot e^T\cdot (\sqrt{1-e^{-1}})^{1-\frac{2}{T}}}
		&=	\underset{T\rightarrow+\infty}{\lim}\left[\dfrac{7\cdot2^{\frac{2}{T}}}{e^{\frac{1}{2}-\frac{1}{2T}}\cdot (\sqrt{1-e^{-1}})^{1-\frac{2}{T}}}\cdot\left(\dfrac{7}{8\cdot e}\right)^T\right]\\
		&=0.
	\end{align*}
	
	We know that
	
	\begin{equation*}
		\underset{T\rightarrow+\infty}{\lim}\circled{2a}^{1/T}=0.
	\end{equation*}
	
	Hence 
	\begin{equation*}
		\underset{T\rightarrow+\infty}{\lim}\circled{2}^{1/T}=0.
	\end{equation*}
	
	Let $G^{(1)}_T:=2\cdot\max\left\{\circled{1}^{1/T}, \circled{2}^{1/T}\right\}$
	
	We showed that 
	\begin{equation*}
		\underset{T\rightarrow+\infty}{\lim}\left(\sum_{n> J/2\pi}P\bigg(\dfrac{1}{TJ}S_n>Kn^{-2}\bigg)\right)^{1/T}\leq	\underset{T\rightarrow+\infty}{\lim}G^{(1)}_T=0
	\end{equation*}
	
	\textbf{Case (2) When $J>2\pi n$} 
	
	\textbf{Step 1.}  For each $n<\frac{J}{2\pi}$, we will find an upper bound of $P\left(\dfrac{1}{TJ}S_T^n>Kn^{-2}\right).$
	
	Let $C_n(t)=\int_{-\infty}^{t}e^{\frac{-1+\omega_n}{2}(t-s)}-e^{\frac{-1-\omega_n}{2}(t-s)}dW_n(s)$, then 
	\begin{equation*}
		C_n(t)=e^{\frac{-1+\omega_n}{2}(t)}\int_{-\infty}^{t}e^{\frac{1-\omega_n}{2}s}dW_n(s)-e^{\frac{-1-\omega_n}{2}(t)}\int_{-\infty}^{t}e^{\frac{1+\omega_n}{2}s}dW_n(s).
	\end{equation*}
	
	Each integral above has mean 0 and variance as the following
	\begin{equation*}
		\begin{split}
			\E\big[\big(\int_{-\infty}^{t}e^{\frac{1-\omega_n}{2}(s)}dW_n(s)\big)^2\big]&=\int_{-\infty}^{t}e^{(1-\omega_n)(s)}dW_n(s)=\dfrac{1}{1-\omega_n}e^{(1-\omega_n)t},\\
			\E\big[\big(\int_{-\infty}^{t}e^{\frac{1+\omega_n}{2}(s)}dW_n(s)\big)^2\big]&=\int_{-\infty}^{t}e^{(1+\omega_n)(s)}dW_n(s)=\dfrac{1}{1+\omega_n}e^{(1+\omega_n)t}.
		\end{split}
	\end{equation*}
	
	Then 
	\begin{equation*}
		\begin{split}
			\int_{-\infty}^{t}e^{\frac{1-\omega_n}{2}\cdots}dW_n(s)
			&\stackrel{D}{=}
			B_{\int_{-\infty}^{t}e^{(1-\omega_n)\cdot s}ds}
			\stackrel{D}{=}
			B_{\frac{1}{1-\omega_n}e^{(1-\omega_n)t}}=:\tilde{B}_t^-,\\
			\int_{-\infty}^{t}e^{\frac{1+\omega_n}{2}\cdot s}dW_n(s)
			&\stackrel{D}{=}
			B_{\int_{-\infty}^{t}e^{(1+\omega_n)\cdot s}ds}
			\stackrel{D}{=}B_{\frac{1}{1+\omega_n}e^{(1+\omega_n)t}}=:\tilde{B}_t^+.
		\end{split}
	\end{equation*}
	
	We define
	\begin{equation*}
		\tilde{B}_t:=\underset{0\leq s\leq t}{\sup}\left|B_{\frac{1}{1-\omega_n}s}\right|.
	\end{equation*}
	
	Since 
	\begin{equation*}
		\tilde{B}_t^\pm\leq\underset{0\leq s\leq t}{\sup}B_{\frac{1}{1-\omega_n}e^{(1+\omega_n)s}}\leq\underset{0\leq s\leq t}{\sup}\left|_{\frac{1}{1-\omega_n}e^{(1+\omega_n)s}}\right|=\tilde{B}_{e^{(1+\omega_n)t}},
	\end{equation*}
	given any $\lambda\in\mathbb{R}$, 
	\begin{equation*}
		\begin{split}
			P(C_n^2(t)>\lambda)
			\leq P\big(4e^{(-1+\omega_n)t}\tilde{B}^2_{e^{(1+\omega_n)t}}>\lambda\big).
		\end{split}
	\end{equation*}
	
	By (\ref{ineqoftildeB}), given $0\leq s\leq t$,
	\begin{equation*}
		\tilde{B}_t\leq 3\tilde{B}_s+\tilde{B}_{s,t}
	\end{equation*}
	where 
	\begin{equation*}
		\tilde{B}_{s,t}:=\underset{s\leq r\leq t}{\sup}\left|B_{\frac{1}{1-\omega_n}r}-B_{\frac{1}{1-\omega_n}s}\right|
	\end{equation*}
	and 
	\begin{equation*}
		\tilde{B}_{s,t}\stackrel{D}{=}\tilde{B}_{t-s}.
	\end{equation*}
	
	For any $\tau>0$, then by Markov inquality,
	\begin{equation}
		\begin{split}
			P\bigg(\dfrac{1}{TJ}S_n>Kn^{-2}\bigg)
			&=P\bigg(\int_{0}^{T}C_n^2(t)dt>\dfrac{K}{n^2}\bigg(\dfrac{\omega_n}{\gamma_n}\bigg)^2TJ\bigg)\\
			&=P\bigg(e^{\tau\int_{0}^{T}C_n^2(t)dt}>e^{\tau\frac{K}{n^2}(\frac{\omega_n}{\gamma_n})^2TJ}\bigg)\\
			&\leq \E\bigg[e^{\tau\int_{0}^{T}C_n^2(t)dt}\bigg]e^{-\tau\frac{K}{n^2}(\frac{\omega_n}{\gamma_n})^2TJ}\\
			&\leq\E\bigg[e^{4\tau\int_{0}^{T}e^{(-1+\omega_n)t}\tilde{B}^2_{e^{(1+\omega_n)t}}dt}\bigg]e^{-\tau\frac{K}{n^2}(\frac{\omega_n}{\gamma_n})^2TJ}.\label{probofJ>2pin}
		\end{split}
	\end{equation}
	
	\textbf{Step 2.} We will find an upper bound of $\int_{0}^{T}e^{(-1+\omega_n)t}\tilde{B}^2_{e^{(1+\omega_n)t}}dt$.
	
	Let $T>2$ be an integer, by Lemma 4.1,
	\begin{equation}
		\begin{split}
			\int_{0}^{T}&e^{(-1+\omega_n)t}\tilde{B}^2_{e^{(1+\omega_n)t}}dt
			\leq\int_{0}^{1}e^{(-1+\omega_n)t}\tilde{B}^2_{e^{(1+\omega_n)t}}dt\\
			&+\tilde{B}^2_{e^{(1+\omega_n)}}\sum_{k=1}^{T-1}
			2\cdot18^k\int_{k}^{k+1}
			e^{(-1+\omega_n)t}dts\\
			& +\sum_{m=1}^{T-2}\bigg[\bigg(\sum_{k=m+1}^{T-1}
			2\cdot18^{k-m}\int_{k}^{k+1}
			e^{(-1+\omega_n)t}\tilde{B}^2_{e^{(1+\omega_n)m},e^{(1+\omega_n)(m+1)}}dt\bigg)\\
			&\qquad\qquad+2\int_{m}^{m+1}e^{(-1+\omega_n)t}\tilde{B}^2_{e^{m(1+\omega_n)},e^{(1+\omega_n)t}}\bigg].\label{J>2pininteB}
		\end{split}
	\end{equation}
	
	Let 
	\begin{equation*}
		\begin{split}
			(\text{I})&=\int_{0}^{1}e^{(-1+\omega_n)t}\tilde{B}^2_{e^{(1+\omega_n)t}}dt
			+\tilde{B}^2_{e^{(1+\omega_n)}}\sum_{k=1}^{T-1}
			2\cdot 18{k}\int_{k}^{k+1}
			e^{(-1+\omega_n)t}dt,\\
			(\text{II})&=\sum_{m=1}^{T-2}\bigg[\bigg(\sum_{k=m+1}^{T-1}
			2\cdot18^{k-m}\int_{k}^{k+1}
			e^{(-1+\omega_n)t}\tilde{B}^2_{e^{(1+\omega_n)m},e^{(1+\omega_n)(m+1)}}dt\bigg)\\
			&\qquad\qquad+2\int_{m}^{m+1}e^{(-1+\omega_n)t}\tilde{B}^2_{e^{m(1+\omega_n)},e^{(1+\omega_n)t}}\bigg].
		\end{split}
	\end{equation*}
	
	We start with (I). Note that (I)$\in\mathcal{F}_{e^{(1+\omega_n)}}$. We have
	\begin{equation*}
		\begin{split}
			\sum_{k=1}^{T-1}2\cdot18^{k}\int_{k}^{k+1}e^{(-1+\omega_n)t}dt
			\leq\dfrac{1}{1-\omega_n}\sum_{k=1}^{T-1}\bigg(\dfrac{18}{e^{1-\omega_n}}\bigg)^k.
		\end{split}
	\end{equation*}
	
	Since $\omega_n=\dfrac{\sqrt{1-4k_n^2}}{2}$ where $k_n=\dfrac{n\pi}{J}$, we have $1-\omega_n>0$, and
	\begin{equation*}
		\sum_{k=1}^{T-1}2\cdot18^{k}\int_{k}^{k+1}e^{(-1+\omega_n)t}dt
		\leq\dfrac{1}{1-\omega_n}\sum_{k=1}^{T-1}18^k=\dfrac{1}{1-\omega_n}\dfrac{18^T-18}{17}.
	\end{equation*}
	
	Then we have 
	\begin{equation*}
		\begin{split}
			(\text{I})
			\leq \int_{0}^{1}e^{(-1+\omega_n)t}\tilde{B}^2_{e^{(1+\omega_n)t}}dt+\dfrac{18^T}{1-\omega_n}\tilde{B}^2_{e^{1+\omega_n}}.
		\end{split}
	\end{equation*}
	
	Since $e^{(-1+\omega_n)t}\leq1$ and $\tilde{B}^2_{e^{(1+\omega_n)t}}\leq\tilde{B}^2_{e^{(1+\omega_n)}}, \forall t\in[0,1]$,
	\begin{equation*}
		(\text{I})\leq \tilde{B}^2_{e^{(1+\omega_n)}}\bigg(1+\dfrac{18^T}{1-\omega_n}\bigg).
	\end{equation*}
	
	Next, we consider 
	\begin{equation*}
		\begin{split}
			(\text{II})&=\sum_{m=1}^{T-2}\bigg[\bigg(\sum_{k=m+1}^{T-1}
			2\cdot18^{k-m}\int_{k}^{k+1}
			e^{(-1+\omega_n)t}\tilde{B}^2_{e^{(1+\omega_n)m},e^{(1+\omega_n)(m+1)}}dt\bigg)\\
			&\qquad\qquad+2\int_{m}^{m+1}e^{(-1+\omega_n)t}\tilde{B}^2_{e^{m(1+\omega_n)},e^{(1+\omega_n)t}}\bigg].
		\end{split}
	\end{equation*}
	
	For each $m\in\{1, 2,\dots, T-2\}$, let
	\begin{equation*}
		\begin{split}
			M_m=&\bigg[\bigg(\sum_{k=m+1}^{T-1}
			2\cdot18^{k-m}\int_{k}^{k+1}
			e^{(-1+\omega_n)t}\tilde{B}^2_{e^{(1+\omega_n)m},e^{(1+\omega_n)(m+1)}}dt\bigg)\\
			&\qquad\qquad+2\int_{m}^{m+1}e^{(-1+\omega_n)t}\tilde{B}^2_{e^{m(1+\omega_n)},e^{(1+\omega_n)t}}\bigg].
		\end{split}
	\end{equation*}
	
	Then $M_m\in\mathcal{F}_{e^{(1+\omega_n)(m+1)}}$ and $M_m\indep\mathcal{F}_{e^{(1+\omega_n)m}}$.
	
	For a fixed $m$, we have 
	\begin{equation*}
		\sum_{k=m+1}^{T-1}2\cdot18^{k-m}\int_{k}^{k+1}e^{(-1+\omega_n)t}\tilde{B}^2_{e^{(1+\omega_n)m},e^{(1+\omega_n)(m+1)}}dt
	\end{equation*}
	\begin{equation*}
		\hspace{1.5in}\leq \dfrac{\tilde{B}^2_{e^{(1+\omega_n)m},e^{(1+\omega_n)(m+1)}}\cdot2\cdot18^{T-m}}{1-\omega_n}.\hspace{3.2cm}
	\end{equation*}
	
	Let $\alpha'_m=2\cdot18^{T-m}$. Then we have
	\begin{equation*}
		\begin{split}
			(\text{II})&\leq \sum_{m=1}^{T-2}\left[\dfrac{\tilde{B}^2_{e^{(1+\omega_n)m},e^{(1+\omega_n)(m+1)}}\cdot\alpha'_m}{1-\omega_n}+2\tilde{B}^2_{e^{(1+\omega_n)m},e^{(1+\omega_n)(m+1)}}\dfrac{e^{(-1+\omega_n)m}}{1-\omega_n}\right]\\
			&=\dfrac{1}{1-\omega_n}\sum_{m=1}^{T-2}\tilde{B}^2_{e^{(1+\omega_n)m},e^{(1+\omega_n)(m+1)}}\left(\alpha'_m+2e^{(-1+\omega_n)m}\right).
		\end{split}
	\end{equation*}
	
	\textbf{Step 3.} We will choose a proper value for $\tau$ and find an explicit upper bound of\\ $E\big[e^{4\tau\int_{0}^{T}e^{(-1+\omega_n)t}\tilde{B}_{e^{(1+\omega_n)t}}^2dt}\big]$.\\
	
	Back to (\ref{J>2pininteB}), we have
	\begin{equation*}
		\begin{split}
			\int_{0}^{T}e^{(-1+\omega_n)t}\tilde{B}^2_{e^{(1+\omega_n)t}}dt
			&\leq\left(\dfrac{18^T}{1-\omega_n}+1\right)\tilde{B}^2_{e^{1+\omega_n}}\\
			&\quad+\dfrac{1}{1-\omega_n}\sum_{m=1}^{T-2}\tilde{B}^2_{e^{(1+\omega_n)m},e^{(1+\omega_n)(m+1)}}\left(\alpha'_m+2e^{(-1+\omega_n)m}\right).
		\end{split}
	\end{equation*}
	
	Then 
	\begin{equation}
		\begin{split}
			\E&\left[e^{4\tau\int_{0}^{T}e^{(-1+\omega_n)t}\tilde{B}^2_{e^{(1+\omega_n)t}}dt}\right]
			\leq\E\left[\exp\left(4\tau\left(\dfrac{18^T}{1-\omega_n}+1\right)\tilde{B}^2_{e^{1+\omega_n}} \right)\right]\\
			&\quad\cdot\prod_{m=1}^{T-2}\E\left[\exp\left(4\tau\dfrac{1}{1-\omega_n}\tilde{B}^2_{e^{(1+\omega_n)m},e^{(1+\omega_n)(m+1)}}\left(\alpha'_m+2e^{(-1+\omega_n)m}\right)\right)\right].\label{J>2pinexpectation}
		\end{split}
	\end{equation}
	
	Let
	\begin{equation*}
		\begin{split}
			\text{(i)}&=\E\left[\exp\left(4\tau\left(\dfrac{18^T}{1-\omega_n}+1\right)\tilde{B}^2_{e^{1+\omega_n}} \right)\right]\\
			\text{(ii)}&=\prod_{m=1}^{T-2}\E\left[\exp\left(4\tau\dfrac{1}{1-\omega_n}\tilde{B}^2_{e^{(1+\omega_n)m},e^{(1+\omega_n)(m+1)}}\left(\alpha'_m+2e^{(-1+\omega_n)m}\right)\right)\right].
		\end{split}
	\end{equation*}
	
	We start with (i).
	\begin{equation*}
		\begin{split}
			\E\bigg[\exp\bigg(4\tau\left(\dfrac{18^T}{1-\omega_n}+1\right)&\tilde{B}^2_{e^{1+\omega_n}}\bigg)\bigg]\\
			&=\int_{1}^{+\infty}P\left(\exp\left(4\tau\left(\dfrac{18^T}{1-\omega_n}+1\right)\tilde{B}^2_{e^{1+\omega_n}}\right)\geq x\right)dx\\
			&=4\int_{1}^{+\infty}P\left(B_{e^{(1+\omega_n)}}\geq\sqrt{\dfrac{\log x}{4\tau\left(\dfrac{18^T}{1-\omega_n}+1\right)}}\right)dx.
		\end{split}
	\end{equation*}
	
	When $1-\dfrac{1}{8\tau e^{1+\omega_n}(\frac{18^T}{1-\omega_n}+1)}<0$, that is $\tau<\dfrac{1}{8\left(\frac{18^T}{1-\omega_n}+1\right)e^{1+\omega_n}}$, by Lemma 4.1, we have
	\begin{equation*}
		\E\left[\exp\left(4\tau\left(\dfrac{18^T}{1-\omega_n}+1\right)\tilde{B}^2_{e^{1+\omega_n}}\right)\right]
		\leq \dfrac{16\tau e^{1+\omega_n}\left(\frac{18^T}{1-\omega_n}+1\right)}{\sqrt{1-8\tau\left(\frac{18^T}{1-\omega_n}+1\right)e^{(1+\omega_n)}}}.
	\end{equation*}
	
	Next we work on (ii).
	
	For each $m\in\{1, 2, \dots, T-2\}$, 
	\begin{equation*}
		\begin{split}
			\E&\left[\exp\left(4\tau\dfrac{1}{1-\omega_n}\tilde{B}^2_{e^{(1+\omega_n)m},e^{(1+\omega_n)(m+1)}}\left(\alpha'_m+2e^{(-1+\omega_n)m}\right)\right)\right]\\
			&=4\int_{1}^{\infty}P\left(B_{e^{(1+\omega_n)(m+1)}-e^{(1+\omega_n)m}}\geq\sqrt{\dfrac{\log x}{4\tau\frac{1}{1-\omega_n}(\alpha'_m+2e^{(-1+\omega_n)m})}}\right)dx.
		\end{split}
	\end{equation*}
	
	Let $\sigma_m^2=e^{(1+\omega_n)(m+1)}-e^{(1+\omega_n)m}$ and $\gamma_m=4\tau\frac{1}{1-\omega_n}(\alpha'_m+2e^{(-1+\omega_n)m})$. When $1-\dfrac{1}{2\gamma_m\sigma_m^2}<0$, that is 
	\begin{center}
		$\tau<\dfrac{1}{8\frac{1}{1-\omega_n}(2\cdot18^{T-m}+2e^{(-1+\omega_n)m})(e^{(1+\omega_n)(m+1)}-e^{(1+\omega_n)m})}$.
	\end{center}
	
	By Lemma 4.2, we have
	\begin{equation*}
		\begin{split}
			\E&\left[\exp\left(4\tau\dfrac{1}{1-\omega_n}\tilde{B}^2_{e^{(1+\omega_n)m},e^{(1+\omega_n)(m+1)}}\left(\alpha'_m+2e^{(-1+\omega_n)m}\right)\right)\right]\\
			&\leq\dfrac{16\tau\frac{1}{1-\omega_n}(\alpha'_m+2e^{(-1+\omega_n)m})\sigma_m^2}{\sqrt{1-8\tau\frac{1}{1-\omega_n}(\alpha'_m+2e^{(-1+\omega_n)m})\sigma_m^2}}.
		\end{split}
	\end{equation*}
	
	Considering (i) and (ii) together, we want $\tau$ to satisfiy the following conditions: $\forall m\in\{1,\dots, T-2\}$,
	\[
	\begin{cases}
		\tau<\dfrac{1}{8\frac{1}{1-\omega_n}(2\cdot18^{T-m}+2e^{(-1+\omega_n)m})(e^{(1+\omega_n)(m+1)}-e^{(1+\omega_n)m})},\\
		\tau<\dfrac{1}{8\left(\frac{18^T}{1-\omega_n}+1\right)e^{1+\omega_n}}.
	\end{cases}
	\]
	
	So, we let
	\begin{equation*}
		\tau=\dfrac{1}{8\cdot\frac{1}{1-\omega_n}\cdot19^T\cdot e^{(1+\omega_n)T}}.
	\end{equation*}
	
	Back to (\ref{J>2pinexpectation}),
	\begin{equation*}
		\begin{split}
			\E\bigg[\exp\bigg(4\tau\int_{0}^{T}e^{(-1+\omega_n)t}&\tilde{B}^2_{e^{(1+\omega_n)t}}dt\bigg)\bigg]\\
			&\leq\dfrac{\frac{16 e^{1+\omega_n}\left(\frac{18^T}{1-\omega_n}+1\right)}{8\cdot\frac{1}{1-\omega_n}\cdot19^T\cdot e^{(1+\omega_n)T}}}{\sqrt{1-8\left(\frac{18^T}{1-\omega_n}+1\right)e^{(1+\omega_n)}\cdot\frac{1}{8\cdot\frac{1}{1-\omega_n}\cdot19^T\cdot e^{(1+\omega_n)T}}}}\\
			&\quad\cdot\prod_{m=1}^{T-2}\dfrac{\frac{16\frac{1}{1-\omega_n}(\alpha'_m+2e^{(-1+\omega_n)m})\sigma_m^2}{8\cdot\frac{1}{1-\omega_n}\cdot19^T\cdot e^{(1+\omega_n)T}}}{\sqrt{1-\frac{1}{19^T\cdot e^{(1+\omega_n)T}}(\alpha'_m+2e^{(-1+\omega_n)m})\sigma_m^2}}.
		\end{split}
	\end{equation*}
	
	Then we simplify the above inequality and get
	\begin{equation*}
		\begin{split}
			\E&\left[\exp\left(4\tau\int_{0}^{T}e^{(-1+\omega_n)t}\tilde{B}^2_{e^{(1+\omega_n)t}}dt\right)\right]\\
			&\leq\dfrac{2\left(\frac{18^T}{1-\omega_n}+1\right)}{\sqrt{\frac{1}{1-\omega_n}\cdot19^T\cdot e^{(1+\omega_n)(T-1)}}\sqrt{\frac{1}{1-\omega_n}\cdot19^T\cdot e^{(1+\omega_n)(T-1)}-\frac{19^T}{1-\omega_n}}}\\
			&\quad\cdot\prod_{m=1}^{T-2}\dfrac{2(\alpha'_m+2e^{(-1+\omega_n)m})\cdot e^{(1+\omega_n)(m+1)}}{\sqrt{19^T\cdot e^{(1+\omega_n)T}}\sqrt{19^T\cdot e^{(1+\omega_n)T}-19^T\cdot e^{(1+\omega_n)(T-1)}}}.
		\end{split}
	\end{equation*}
	
	Since $0<\omega_n<1$, as $T>5$, we have the following
	\begin{equation*}
		\begin{split}
			\E&\left[\exp\left(4\tau\int_{0}^{T}e^{(-1+\omega_n)t}\tilde{B}^2_{e^{(1+\omega_n)t}}dt\right)\right]\\
			&\leq \dfrac{2(18^T+1-\omega_n)}{19^T\sqrt{e^{(1+\omega_n)(2T-2)}-e^{(1+\omega_n)(T-1)}}}\prod_{m=1}^{T-2}\dfrac{2\left(2\cdot18^Te^{(1+\omega_n)}\frac{e^{(1+\omega_n)m}}{18^m}+2e^{(1+\omega_n)}\cdot e^{2\omega_nm}\right)}{19^T\sqrt{e^{(1+\omega_n)2T}-e^{(1+\omega_n)(2T-1)}}}.
		\end{split}
	\end{equation*}
	
	Since $\frac{e^{(1+\omega_n)m}}{18^m}<1$ and $e^{2\omega_nm}<18^T$ as $T$ large enough, 
	\begin{equation*}
		\begin{split}
			\E&\left[\exp\left(4\tau\int_{0}^{T}e^{(-1+\omega_n)t}\tilde{B}^2_{e^{(1+\omega_n)t}}dt\right)\right]\\
			&\leq \dfrac{4\cdot18^T}{19^T\sqrt{e^{(1+\omega_n)(2T-2)}-e^{(1+\omega_n)(T-1)}}}\cdot\dfrac{8^{T-2}\cdot e^{(1+\omega_n)(T-2)}\cdot 18^{T(T-2)} }{19^{T(T-2)}e^{(1+\omega_n)T(T-2)}\left(\sqrt{1-e^{-(1+\omega_n)}}\right)^{T-2}}.
		\end{split}
	\end{equation*}
	
	\textbf{Step 4.} We are ready to get an upper bound of $P\left(\dfrac{1}{TJ}S_n>Kn^{-2}\right)$ and  will show that 
	\begin{equation*}
		\underset{T\rightarrow+\infty}{\lim}P\left(\dfrac{1}{TJ}S_n>Kn^{-2}\right)^{1/T}=0.
	\end{equation*}
	
	Back to (\ref{probofJ>2pin}), we have
	\begin{equation*}
		\begin{split}
			P\left(\dfrac{1}{TJ}S_n>Kn^{-2}\right)
			&\leq\dfrac{4\cdot18^T\cdot 8^{T-2}\cdot e^{(1+\omega_n)(T-2)}}{19^T\sqrt{e^{(1+\omega_n)(2T-2)}-e^{(1+\omega_n)(T-1)}}}\\
			&\quad\cdot\dfrac{18^{T(T-2)}}{19^{T(T-2)}e^{(1+\omega_n)T(T-2)}\left(\sqrt{1-e^{-(1+\omega_n)}}\right)^{T-2}}\\
			&\quad\cdot\dfrac{1}{\exp\left(\frac{1}{8\frac{1}{1-\omega_n}19^Te^{(1+\omega_n)T}}\frac{K}{n^2}\left(\frac{\omega_n}{\gamma_n}\right)^2TJ\right)}.
		\end{split}
	\end{equation*}
	

	For $0<n<\frac{J}{2\pi}$, let
	\begin{equation*}
		\text{(1)}_n=\dfrac{18^{T(T-2)}}{19^{T(T-1)}},\quad \text{(2)}_n=\dfrac{4\cdot18^T\cdot 8^{T-2}\cdot e^{(1+\omega_n)(T-2)}}{e^{(1+\omega_n)T(T-2)}},
	\end{equation*}
	\begin{equation*}
		\text{(3)}_n=\dfrac{1}{\sqrt{e^{(1+\omega_n)(2T-2)}-e^{(1+\omega_n)(T-1)}}\cdot\left(\sqrt{1-e^{-(1+\omega_n)}}\right)^{T-2}},
	\end{equation*}
	and
	\begin{equation*}
		\text{(4)}_n=\dfrac{1}{\exp\left(\frac{1}{8\frac{1}{1-\omega_n}19^Te^{(1+\omega_n)T}}\frac{K}{n^2}\left(\frac{\omega_n}{\gamma_n}\right)^2TJ\right)}.
	\end{equation*}
	
	Then
	\begin{equation*}
		P\left(\dfrac{1}{TJ}S_n>Kn^{-2}\right)\leq \text{(1)}_n\cdot\text{(2)}_n\cdot\text{(3)}_n\cdot\text{(4)}_n.
	\end{equation*}
	
	For fixed $n$ such that $0<n<\dfrac{J}{2\pi}$,
	\begin{equation*}
		\underset{T\rightarrow+\infty}{\lim}\text{(4)}_n^{1/T}=	\underset{T\rightarrow+\infty}{\lim}\dfrac{1}{\exp\left(\frac{1}{8\frac{1}{1-\omega_n}19^Te^{(1+\omega_n)T}}\frac{K}{n^2}\left(\frac{\omega_n}{\gamma_n}\right)^2J\right)}=\dfrac{1}{e^0}=1.
	\end{equation*}
	
	Then for $(3)_n$,
	\begin{equation*}
		\begin{split}
			(3)_n^{1/T}&=\dfrac{1}{\left(\sqrt{e^{(1+\omega_n)(2T-2)}-e^{(1+\omega_n)(T-1)}}\right)^{\frac{1}{T}}\cdot \left(\sqrt{1-e^{-(1+\omega_n)}}\right)^{1-\frac{2}{T}}}\\
			&\leq\dfrac{1}{e^{(1+\omega_n)(\frac{1}{2}-\frac{1}{2T})}\cdot \left(\sqrt{1-e^{-(1+\omega_n)}}\right)^{1-\frac{2}{T}}}.
		\end{split}
	\end{equation*}
	
	We have
	\begin{equation*}
		\underset{T\rightarrow+\infty}{\lim}(3)_n^{1/T}=	\dfrac{1}{e^{(1+\omega_n)\cdot\frac{1}{2}}\cdot \left(\sqrt{1-e^{-(1+\omega_n)}}\right)}.
	\end{equation*}
	
	Also
	\begin{equation*}
		\underset{T\rightarrow+\infty}{\lim}(1)_n^{1/T}=\underset{T\rightarrow+\infty}{\lim}\dfrac{18^{T-1}}{19^{T-1}}=0.
	\end{equation*}
	
	Finally,
	\begin{equation*}
		\underset{T\rightarrow+\infty}{\lim}(2)_n^{1/T}
		=\underset{T\rightarrow+\infty}{\lim}\dfrac{8^{1-\frac{2}{T}}\cdot e^{(1+\omega_n)(1-\frac{2}{T})}}{e^{(1+\omega_n)(T-2)}}=0.
	\end{equation*}
	
	Let $G^{(2)}_T=\dfrac{J}{2\pi}\cdot \left(\underset{0<n<\frac{J}{2\pi}}{\max}\left\{(1)_n\cdot(2)_n\cdot(3)_n\cdot(4)_n\right\}\right)^{1/T}.$
	
	We showed that 
	\begin{equation*}
		\underset{T\rightarrow+\infty}{\lim}\left(\sum_{0<n<\frac{J}{2\pi}}P\left(\dfrac{1}{TJ}S_n>Kn^{-2}\right)\right)^{1/T}\leq\underset{T\rightarrow+\infty}{\lim}G^{(2)}_T=0.
	\end{equation*}

	
	\textbf{Case (3) When $J=2\pi n$,}
	recall 
	\begin{equation*}
		\begin{split}
			\tilde{a}_n(t)
			&=\int_{-\infty}^{t}e^{-\frac{1}{2}(t-s)}\gamma_n(t-s)dW_n(s)\\
			&=\gamma_n\int_{-\infty}^{t}e^{-\frac{1}{2}(t-s)}(t-s)dW_n(s).
		\end{split}
	\end{equation*}
	
	\textbf{Step 1.}  We will find an upper bound of $P\left(\dfrac{1}{TJ}S_T^n>Kn^{-2}\right)$ when  $n<\frac{J}{2\pi}$.
	
	Let
	\begin{equation*}
		D_n(t)=\int_{-\infty}^{t}e^{-\frac{1}{2}(t-s)}(t-s)dW_n(s),
	\end{equation*}
	then 
	\begin{equation}
		\begin{split}
			D_n(t)
			&=e^{-\frac{1}{2}t}\int_{-\infty}^{t}e^{\frac{1}{2}s}(t-s)dW_n(s)\\
			&=te^{-\frac{1}{2}t}\int_{-\infty}^{t}e^{\frac{1}{2}s}dW_n(s)-e^{-\frac{1}{2}t}\int_{-\infty}^{t}s\cdot e^{\frac{1}{2}s}dW_n(s).\label{covJ=2pin}
		\end{split}
	\end{equation}
	
	The covariance of each Ito's itegral from (\ref{covJ=2pin}):
	\begin{equation*}
		\E\left[\left(\int_{-\infty}^{t}e^{\frac{1}{2}s}dW_n(s)\right)^2\right]
		=\int_{-\infty}^{t}e^sds=e^t.
	\end{equation*}
	and 
	\begin{equation*}
		\begin{split}
			\E\left[\left(\int_{-\infty}^{t}s\cdot e^{\frac{1}{2}s}dW_n(s)\right)^2\right]
			=e^t(t-1)^2.
		\end{split}
	\end{equation*}
	
	Next,  we let
	\begin{equation*}
		I_t^{(1)}:=\int_{-\infty}^{t}e^{\frac{1}{2}s}dW_n(s)
	\end{equation*}
	and 
	\begin{equation*}
		I_t^{(2)}:=\int_{-\infty}^{t}s\cdot e^{\frac{1}{2}s}dW_n(s).
	\end{equation*}
	
	Then 
	\begin{equation*}
		\begin{split}
			D_n(t)&=te^{-\frac{1}{2}t}I_t^{(1)}-e^{-\frac{1}{2}t}I_t^{(2)}\\
			I_t^{(1)}&\stackrel{D}{=}B_{e^t}\\
			I_t^{(2)}&\stackrel{D}{=}B_{e^t(t-1)^2}.
		\end{split}
	\end{equation*}
	
	Since
	\begin{equation*}
		e^t+e^t(t-1)^2\leq e^t+e^{2t}\leq 2e^{2t},
	\end{equation*}
	we define
	\begin{equation*}
		\tilde{B}_t:=\underset{0\leq s\leq t}{\sup}\left|B_{2s}\right|,
	\end{equation*}
	then for each $i\in\{1, 2\}$, we have
	\begin{equation*}
		\left|I_t^{(i)}\right|\leq\underset{0\leq s\leq t}{\sup}\left|B_{2e^{2t}}\right|=\tilde{B}_{e^{2t}}.
	\end{equation*}
	
	Recall
	\begin{equation*}
		\begin{split}
			S_n(T)
			&=\int_{0}^{T}\left(\tilde{a}_n(t)\right)^2dt\\
			&=\gamma_n^2\int_{0}^{T}D_n^2(t)dt.
		\end{split}
	\end{equation*}
	
	Let $\tau>0$,
	\begin{equation}
		\begin{split}
			P\left(\dfrac{1}{TJ}S_n>Kn^{-2}\right)
			\leq\E\big[\exp\left(\tau\int_{0}^{T}D_n^2(t)dt\right)\big]\cdot \exp\left(-\dfrac{\tau KTJ}{n^2\gamma_n^2}\right).\label{markovJ=2pin}
		\end{split}
	\end{equation}
	
	Since $\exp\left(\tau\int_{0}^{T}D_n^2(t)dt\right)$ is a nonnegative random variable, 
	
	\begin{equation}
		\begin{split}
			\E\left[\exp\left(\tau\int_{0}^{T}D_n^2(t)dt\right)\right]
			=\int_{0}^{+\infty}P\left(\int_{0}^{T}\left(te^{-\frac{1}{2}t}I_t^{(1)}-e^{-\frac{1}{2}t}I_t^{(2)}\right)^2dt>\dfrac{\log s}{\tau}\right)ds.\label{markovJ=2pin-1}
		\end{split}	
	\end{equation}
	
	By Cauchy-Schwarz inequality, 
	\begin{equation}
		\begin{split}
			\E\left[\exp\left(\tau\int_{0}^{T}D_n^2(t)dt\right)\right]
			\leq \int_{0}^{+\infty}2P\left(\int_{0}^{T}2(t^2+1)e^{-t}\left(\tilde{B}_{e^{2t}}\right)^2dt>\dfrac{\log s}{2\tau}\right)ds.
		\end{split}\label{markovJ=2pin-2}
	\end{equation}
	
	\begin{lemma}
		$\exp\left(\dfrac{9}{10}t\right)\leq t^2+1$, for all $t\in\mathbb{R}$.
	\end{lemma}
	The proof of Lemma 4.3 is in Appendix \ref{exp9/10t}.

	By Lemma 4.3, we have
	\begin{equation*}
		\begin{split}
			P\left(\int_{0}^{T}2(t^2+1)e^{-t}\left(\tilde{B}_{e^{2t}}\right)^2dt>\dfrac{\log s}{2\tau}\right)
			\leq P\left(4\tau\int_{0}^{T}e^{-\frac{1}{10}t}\left(\tilde{B}_{e^{2t}}\right)^2dt>\log s\right).
		\end{split}
	\end{equation*}
	
	Back to (\ref{markovJ=2pin-1}) and (\ref{markovJ=2pin-2}), we have
	\begin{equation*}
		\begin{split}
			\E\left[\exp\left(\tau\int_{0}^{T}D_n^2(t)dt\right)\right]
			\leq 2\E\left[\exp\left(4\tau\int_{0}^{T}e^{-\frac{1}{10}t}\left(\tilde{B}_{e^{2t}}\right)^2dt\right)\right].
		\end{split}
	\end{equation*}
	
	Then continuing from (\ref{markovJ=2pin}), 
	\begin{equation}
		P\left(\dfrac{1}{TJ}S_n>Kn^{-2}\right)
		\leq2\E\left[\exp\left(4\tau\int_{0}^{T}e^{-\frac{1}{10}t}\left(\tilde{B}_{e^{2t}}\right)^2dt\right)\right]\cdot \exp\left(-\dfrac{\tau KTJ}{n^2\gamma_n^2}\right).\label{markovJ=2pin_2}
	\end{equation}
	
	\textbf{Step 2.} We will find an upper bound of $\int_{0}^{T}e^{-\frac{1}{10}t}\tilde{B}^2_{e^{2t}}dt$.

	Recall 
	\begin{equation*}
		\tilde{B}_t:=\underset{0\leq s\leq t}{\sup}\left|B_{2s}\right|.
	\end{equation*}
	
	Since $n$ is fixed, by (\ref{ineqoftildeB}), given $0\leq s\leq t$,
	\begin{equation*}
		\tilde{B}_t\leq 3\tilde{B}_s+\tilde{B}_{s,t}
	\end{equation*}
	where 
	\begin{equation*}
		\tilde{B}_{s,t}:=\underset{s\leq r\leq t}{\sup}\left|B_{2r}-B_{2s}\right|
	\end{equation*}
	and 
	\begin{equation*}
		\tilde{B}_{s,t}\stackrel{D}{=}\tilde{B}_{t-s}.
	\end{equation*}
	
	Let $T>2$ be an integer, by Lemma 4.1,
	\begin{equation}
		\begin{split}
			\int_{0}^{T}e^{-\frac{1}{10}t}\tilde{B}^2_{e^{2t}}dt
			&=\int_{0}^{1}e^{-\frac{1}{10}t}\tilde{B}^2_{e^{2t}}dt
			+\tilde{B}^2_{e^{2}}\sum_{k=1}^{T-1}18^{k}\int_{k}^{k+1}e^{-\frac{1}{10}t}dts\\
			&\quad +\sum_{m=1}^{T-2}\bigg[\bigg(\sum_{k=m+1}^{T-1}2\cdot18^{k-m}\int_{k}^{k+1}e^{-\frac{1}{10}t}\tilde{B}^2_{e^{2m},e^{2(m+1)}}dt\bigg)\\
			&\quad+2\int_{m}^{m+1}e^{-\frac{1}{10}t}\tilde{B}^2_{e^{2m},e^{2t}}\bigg].\label{J=2pininteB}
		\end{split}
	\end{equation}
	
	Let 
	\begin{equation*}
		\begin{split}
			(\text{I})&=\int_{0}^{1}e^{-\frac{1}{10}t}\tilde{B}^2_{e^{2t}}dt
			+\tilde{B}^2_{e^{2}}\sum_{k=1}^{T-1}18^{k}\int_{k}^{k+1}e^{-\frac{1}{10}t}dts\\
			(\text{II})&=\sum_{m=1}^{T-2}\bigg[\bigg(\sum_{k=m+1}^{T-1}2\cdot18^{k-m}\int_{k}^{k+1}e^{-\frac{1}{10}t}\tilde{B}^2_{e^{2m},e^{2(m+1)}}dt\bigg)
			+2\int_{m}^{m+1}e^{-\frac{1}{10}t}\tilde{B}^2_{e^{2m},e^{2t}}\bigg].
		\end{split}
	\end{equation*}
	
	We start with (I). Note that (I)$\in\mathcal{F}_{e^2}$.
	\begin{equation*}
		\begin{split}
			\sum_{k=1}^{T-1}18^{k}\int_{k}^{k+1}e^{-\frac{1}{10}t}dt
			\leq10\sum_{k=1}^{T-1}\bigg(\dfrac{18}{e^{\frac{1}{10}}}\bigg)^k.
		\end{split}
	\end{equation*}
	
	Then we have 
	\begin{equation*}
		\sum_{k=1}^{T-1}18^{k}\int_{k}^{k+1}e^{-\frac{1}{10}t}dt
		\leq10\sum_{k=1}^{T-1}18^k=10\cdot\dfrac{18^T-18}{17}
	\end{equation*}
	and
	\begin{equation*}
		\begin{split}
			(\text{I})
			\leq \int_{0}^{1}e^{-\frac{1}{10}t}\tilde{B}^2_{e^{2t}}dt+10\cdot18^T\tilde{B}^2_{e^{2}}.
		\end{split}
	\end{equation*}
	
	Since $e^{-\frac{1}{10}t}\leq1$ and $\tilde{B}^2_{e^{2t}}\leq\tilde{B}^2_{e^{2}}, \forall t\in[0,1]$,
	\begin{equation*}
		(\text{I})\leq \tilde{B}^2_{e^{2}}\bigg(1+10\cdot18^T\bigg).
	\end{equation*}
	
	Next, we consider 
	\[
	(\text{II})=\sum_{m=1}^{T-2}\bigg[\bigg(\sum_{k=m+1}^{T-1}2\cdot18^{k-m}\int_{k}^{k+1}e^{-\frac{1}{10}t}\tilde{B}^2_{e^{2m},e^{2(m+1)}}dt\bigg)
	+2\int_{m}^{m+1}e^{-\frac{1}{10}t}\tilde{B}^2_{e^{2m},e^{2t}}\bigg].
	\]

	For each $m\in\{1, 2,\dots, T-2\}$, let
	\[
	M_m=\bigg[\bigg(\sum_{k=m+1}^{T-1}2\cdot18^{k-m}\int_{k}^{k+1}e^{-\frac{1}{10}t}\tilde{B}^2_{e^{2m},e^{2(m+1)}}dt\bigg)
	+2\int_{m}^{m+1}e^{-\frac{1}{10}t}\tilde{B}^2_{e^{2m},e^{2t}}\bigg].
	\]

	Then $M_m\in\mathcal{F}_{e^{2(m+1)}}$ and $M_m\indep\mathcal{F}_{e^{2m}}$.
	
	\begin{equation*}
			\sum_{k=m+1}^{T-1}2\cdot18^{k-m}\int_{k}^{k+1}e^{-\frac{1}{10}t}\tilde{B}^2_{e^{2m},e^{2(m+1)}}dt
			\leq 10\tilde{B}^2_{e^{2m},e^{2(m+1)}}\cdot2\cdot18^{T-m}.
	\end{equation*}
	
	We have
	\begin{equation*}
		\begin{split}
			(\text{II})&\leq \sum_{m=1}^{T-2}\left[10\tilde{B}^2_{e^{2m},e^{2(m+1)}}\cdot2\cdot18^{T-m}+2\tilde{B}^2_{e^{2m},e^{2(m+1)}}\cdot10e^{-\frac{1}{10}m}\right]\\
			&=10\sum_{m=1}^{T-2}\tilde{B}^2_{e^{2m},e^{2(m+1)}}\left(2\cdot18^{T-m}+2e^{-\frac{1}{10}m}\right).
		\end{split}
	\end{equation*}
	
	Going back to (\ref{J=2pininteB}), we have
	\[
	\int_{0}^{T}e^{-\frac{1}{10}t}\tilde{B}^2_{e^{2t}}dt
	\leq\left(10\cdot18^T+1\right)\tilde{B}^2_{e^{2}}\\
	+10\sum_{m=1}^{T-2}\tilde{B}^2_{e^{2m},e^{2(m+1)}}\left(2\cdot18^{T-m}+2e^{-\frac{1}{10}m}\right).
	\]

	\textbf{Step 3.} We will choose a proper value for $\tau$ and find an explicit upper bound of\\ $E\big[e^{4\tau\int_{0}^{T}e^{-\frac{1}{10}t}\tilde{B}^2_{e^{2t}}dt}\big]$.\\
	
	From (\ref{markovJ=2pin_2})
	\begin{equation}
		\begin{split}
			\E&\left[e^{4\tau\int_{0}^{T}e^{-\frac{1}{10}t}\tilde{B}^2_{e^{2t}}dt}\right]
			\leq\E\left[\exp\left(4\tau\left(10\cdot18^T+1\right)\tilde{B}^2_{e^{2}} \right)\right]\\
			&\quad\cdot\prod_{m=1}^{T-2}\E\left[\exp\left(4\tau\cdot10\cdot\tilde{B}^2_{e^{2m},e^{2(m+1)}}\left(2\cdot18^{T-m}+2e^{-\frac{1}{10}m}\right)\right)\right].\label{318}
		\end{split}
	\end{equation}
	
	Let
	\begin{equation*}
		\begin{split}
			\text{(i)}&=\E\left[\exp\left(4\tau\left(10\cdot18^T+1\right)\tilde{B}^2_{e^{2}} \right)\right]\\
			\text{(ii)}&=\prod_{m=1}^{T-2}\E\left[\exp\left(4\tau\cdot10\cdot\tilde{B}^2_{e^{2m},e^{2(m+1)}}\left(2\cdot18^{T-m}+2e^{-\frac{1}{10}m}\right)\right)\right].
		\end{split}
	\end{equation*}
	
	We start with (i).
	\begin{equation*}
		\begin{split}
			\E\left[\exp\left(4\tau\left(10\cdot18^T+1\right)\tilde{B}^2_{e^{2}}\right)\right]
			&=\int_{1}^{+\infty}P\left(\exp\left(4\tau\left(10\cdot18^T+1\right)\tilde{B}^2_{e^{2}}\right)\geq x\right)dx\\
			&=4\int_{1}^{+\infty}P\left(B_{e^{2}}\geq\sqrt{\dfrac{\log x}{4\tau\left(10\cdot18^T+1\right)}}\right)dx.
		\end{split}
	\end{equation*}
	
	When $1-\dfrac{1}{8e^2\tau\left(10\cdot18^T+1\right)}<0$, that is $\tau<\dfrac{1}{8\left(10\cdot18^T+1\right)e^{2}}$, by Lemma 4.2, 
	\begin{equation*}
		\E\left[\exp\left(4\tau\left(10\cdot18^T+1\right)\tilde{B}^2_{e^{2}}\right)\right]\leq \dfrac{16\tau e^{2}\left(10\cdot18^T+1\right)}{\sqrt{1-8\tau\left(10\cdot18^T+1\right)e^{2}}}.
	\end{equation*}
	
	Next, we work on (ii). Let $\alpha'_m=2\cdot18^{T-m}$ and $\sigma_m^2=e^{2(m+1)}-e^{2m}$.
	
	For each $m\in\{1, 2, \dots, T-2\}$, 
	\begin{equation*}
		\begin{split}
			\E&\left[\exp\left(4\tau\cdot10\cdot\tilde{B}^2_{e^{2m},e^{2(m+1)}}\left(2\cdot18^{T-m}+2e^{-\frac{1}{10}m}\right)\right)\right]\\
			&=4\int_{1}^{\infty}P\left(B_{e^{2(m+1)}-e^{2m}}\geq\sqrt{\dfrac{\log x}{4\tau\cdot10\cdot(\alpha'_m+2e^{-\frac{1}{10}m})}}\right)dx.
		\end{split}
	\end{equation*}
	
	When $1-\dfrac{1}{2\sigma_m^2(4\tau\cdot10\cdot(\alpha'_m+2e^{-\frac{1}{10}m}))}<0$, that is 
	\begin{equation*}
		\tau<\dfrac{1}{8\cdot10\cdot(2\cdot18^{T-m}+2e^{-\frac{1}{10}m})(e^{2(m+1)}-e^{2m})},
	\end{equation*}
	by lemma 4.2, 
	\begin{equation*}
		\begin{split}
			\E&\left[\exp\left(4\tau\cdot10\cdot\tilde{B}^2_{e^{2m},e^{2(m+1)}}\left(2\cdot18^{T-m}+2e^{-\frac{1}{10}m}\right)\right)\right]\\
			&\leq\dfrac{16\tau\cdot10\cdot(\alpha'_m+2e^{-\frac{1}{10}m})\sigma_m^2}{\sqrt{1-8\tau\frac{1}{1-\omega_n}(\alpha'm+2e^{-\frac{1}{10}m})\sigma_m^2}}.
		\end{split}
	\end{equation*}
	
	By considering (i) and (ii), we want the $\tau$ to satisfiy the following conditions:
	\[
	\begin{cases}
		\tau<\dfrac{1}{8\cdot10\cdot(2^{1-m}\cdot9^{-m}\cdot18^T+2e^{-\frac{1}{10}m})(e^{2(m+1)}-e^{2m})}, & \forall m\in\{1,\dots, T-2\}\\
		\tau<\dfrac{1}{8\left(10\cdot18^T+1\right)e^{2}}.
	\end{cases}
	\]
	
	Then we let 
	\begin{equation*}
		\tau=\dfrac{1}{8\cdot10\cdot19^T\cdot e^{2T}}.
	\end{equation*}
	
	Back to (\ref{318}), we get
	\begin{equation*}
		\begin{split}
			\E\left[\exp\left(4\tau\int_{0}^{T}e^{-\frac{1}{10}t}\tilde{B}^2_{e^{2t}}dt\right)\right]
			&\leq\dfrac{16\tau e^{2}\left(10\cdot18^T+1\right)}{\sqrt{1-8\tau\left(10\cdot18^T+1\right)e^{2}}}\\
			&\quad\cdot\prod_{m=1}^{T-2}\dfrac{16\tau\cdot10\cdot(\alpha'_m+2e^{-\frac{1}{10}m})\sigma_m^2}{\sqrt{1-8\tau\cdot10\cdot(\alpha'_m+2e^{-\frac{1}{10}m})\sigma_m^2}}.
		\end{split}
	\end{equation*}
	
	Now we replace $\tau$ by its value, 
	\begin{equation}
		\begin{split}
			\E\left[\exp\left(4\tau\int_{0}^{T}e^{-\frac{1}{10}t}\tilde{B}^2_{e^{2t}}dt\right)\right]
			&\leq \dfrac{\frac{16 e^{2}\left(10\cdot18^T+1\right)}{8\cdot10\cdot19^T\cdot e^{2T}}}{\sqrt{1-8\left(10\cdot18^T+1\right)e^{2}\cdot\frac{1}{8\cdot10\cdot19^T\cdot e^{2T}}}}\\
			&\quad\cdot\prod_{m=1}^{T-2}\dfrac{\frac{16\cdot10\cdot(\alpha'_m+2e^{-\frac{1}{10}m})\sigma_m^2}{8\cdot10\cdot19^T\cdot e^{2T}}}{\sqrt{1-\frac{1}{19^T\cdot e^{2T}}(\alpha'_m+2e^{-\frac{1}{10}m})\sigma_m^2}}. \label{J=2pin2}
		\end{split}
	\end{equation}
	
	We simplify the right hand side of (\ref{J=2pin2}), 
	\begin{equation*}
		\begin{split}
			\E&\left[\exp\left(4\tau\int_{0}^{T}e^{-\frac{1}{10}t}\tilde{B}^2_{e^{2t}}dt\right)\right]\\
			&\leq \dfrac{2\left(10\cdot18^T+1\right)}{\sqrt{10\cdot19^T\cdot e^{2(T-1)}}\sqrt{10\cdot19^T\cdot e^{2(T-1)}-\left(10\cdot18^T+1\right)}}\\
			&\quad\cdot\prod_{m=1}^{T-2}\bigg[\dfrac{2(\alpha'_m+2e^{-\frac{1}{10}m})(e^{2(m+1)}-e^{2m})}{\sqrt{19^T\cdot e^{2T}}}\\
			&\qquad\cdot \dfrac{1}{\sqrt{19^T\cdot e^{2T}-(2^{1-m}\cdot9^{-m}\cdot18^T+2e^{-\frac{1}{10}m})(e^{2(m+1)}-e^{2m})}}\bigg].
		\end{split}
	\end{equation*}
	
	Since when $T\geq3$, $\left(10\cdot18^T+1\right)<10\cdot19^T$ and $\forall m\in\{1,\dots, T-2\}$,\\ $(2^{1-m}\cdot9^{-m}\cdot18^T+2e^{-\frac{1}{10}m})(e^{2(m+1)}-e^{2m})<19^T\cdot e^{2(T-1)}$,
	\begin{equation*}
		\begin{split}
			\E\bigg[\exp\bigg(&4\tau\int_{0}^{T}e^{-\frac{1}{10}t}\tilde{B}^2_{e^{2t}}dt\bigg)\bigg]\\
			&\leq	
			\dfrac{2\left(10\cdot18^T+1\right)}{\sqrt{10\cdot19^T\cdot e^{2(T-1)}}\sqrt{10\cdot19^T\cdot e^{2(T-1)}-10\cdot19^T}}\\
			&\quad\cdot\prod_{m=1}^{T-2}\dfrac{2(\alpha'_m+2e^{-\frac{1}{10}m})\cdot e^{2(m+1)}}{\sqrt{19^T\cdot e^{2T}}\sqrt{19^T\cdot e^{2T}-19^T\cdot e^{2(T-1)}}}\\
			&\leq	
			\dfrac{2\cdot10(18^T+\frac{1}{10})}{10\cdot19^T\sqrt{e^{2(2T-2)}-e^{2(T-1)}}}\\
			&\quad\cdot\prod_{m=1}^{T-2}\dfrac{2(\alpha'_m e^{2(m+1)}+2e^{(2-\frac{1}{10})m+2})}{19^T\sqrt{e^{2\cdot2T}-e^{2\cdot(2T-1)}}}	.
		\end{split}
	\end{equation*}
	
	Since $\forall m\in\{1,\dots, T-2\}$, 
	\begin{equation*}
		\alpha'_m=2^{1-m}\cdot9^{-m}\cdot18^T=2\cdot\dfrac{18^T}{18^m}.
	\end{equation*}
	
	We have
	\begin{equation*}
		\begin{split}
			\E&\left[\exp\left(4\tau\int_{0}^{T}e^{-\frac{1}{10}t}\tilde{B}^2_{e^{2t}}dt\right)\right]\\
			&\leq	\dfrac{2(18^T+\frac{1}{10})}{19^T\sqrt{e^{2(2T-2)}-e^{2(T-1)}}}
			\cdot\prod_{m=1}^{T-2}\dfrac{2\left(2\cdot18^Te^{2}\frac{e^{2m}}{18^m}+2e^{2}\cdot e^{(2-\frac{1}{10})m}\right)}{19^T\sqrt{e^{2\cdot2T}-e^{2\cdot(2T-1)}}}.
		\end{split}
	\end{equation*}
	
	Since $\dfrac{e^{2m}}{18^m}<1$, and  $e^{(2-\frac{1}{10})}m<18^T$, $\forall m\in\{1,\dots, T-2\}$. Then
	\begin{equation*}
		\begin{split}
			\E&\left[\exp\left(4\tau\int_{0}^{T}e^{-\frac{1}{10}t}\tilde{B}^2_{e^{2t}}dt\right)\right]\\
			&\leq\dfrac{2(18^T+\frac{1}{10})}{19^T\sqrt{e^{2(2T-2)}-e^{2(T-1)}}}\prod_{m=1}^{T-2}\dfrac{2\cdot2\cdot\left(2\cdot e^{2}\cdot18^T\right)}{19^Te^{2T}\sqrt{1-e^{-2}}}\\
			&\leq\dfrac{4\cdot18^T}{19^T\sqrt{e^{2(2T-2)}-e^{2(T-1)}}}\cdot\dfrac{8^{T-2}\cdot e^{2(T-2)}\cdot 18^{T(T-2)} }{19^{T(T-2)}e^{2T(T-2)}\left(\sqrt{1-e^{-2}}\right)^{T-2}}.
		\end{split}
	\end{equation*}
	
	\textbf{Step 4.} We are ready to get an upper bound of $P\left(\dfrac{1}{TJ}S_n>Kn^{-2}\right)$ and will show that 
	\begin{equation*}
		\underset{T\rightarrow+\infty}{\lim}P\left(\dfrac{1}{TJ}S_n>Kn^{-2}\right)^{1/T}=0.
	\end{equation*}
	
	Back to (\ref{markovJ=2pin_2}), 
	\begin{equation*}
		\begin{split}
			P\left(\dfrac{1}{TJ}S_n>Kn^{-2}\right)
			&\leq\dfrac{4\cdot18^T\cdot 8^{T-2}\cdot e^{2(T-2)}\cdot 18^{T(T-2)}}{19^T\sqrt{e^{2(2T-2)}-e^{2(T-1)}}}\\
			&\quad\cdot\dfrac{1}{19^{T(T-2)}e^{2T(T-2)}\left(\sqrt{1-e^{-2}}\right)^{T-2}}\\
			&\quad\cdot\dfrac{1}{\exp\left(\frac{1}{8\cdot10\cdot19^Te^{2T}}\frac{K}{n^2\gamma_n^2}TJ\right)}\\
			&=\dfrac{4\cdot18^T\cdot 8^{T-2}\cdot e^{2(T-2)}\cdot 18^{T(T-2)}}{19^{T^2-T}e^{2T(T-2)}}\\
			&\quad\cdot\dfrac{1}{\sqrt{e^{2(2T-2)}-e^{2(T-1)}}\cdot\left(\sqrt{1-e^{-2}}\right)^{T-2}}\\
			&\quad\cdot\dfrac{1}{\exp\left(\frac{1}{8\cdot10\cdot19^Te^{2T}}\frac{K}{n^2\gamma_n^2}TJ\right)}\\
			&=:(1)\cdot(2)\cdot(3).
		\end{split}
	\end{equation*}
	
	Since $n=\dfrac{J}{2\pi}$,  
	
	\begin{equation*}
		\underset{T\rightarrow+\infty}{\lim}(3)^{1/T}=\underset{T\rightarrow+\infty}{\lim}\dfrac{1}{\exp\left(\frac{1}{8\cdot10\cdot19^Te^{2T}}\frac{K}{n^2\gamma_n^2}J\right)}=\dfrac{1}{e^0}=1.
	\end{equation*}
	
	Next, we have
	\begin{equation*}
		\dfrac{1}{\sqrt{e^{2(2T-2)}-e^{2(T-1)}}\cdot \left(\sqrt{1-e^{-2}}\right)^{T-2}}\leq\dfrac{1}{\sqrt{e^{2(T-1)}}\cdot \left(\sqrt{1-e^{-2}}\right)^{T-2}}
	\end{equation*}
	
	Then 
	\begin{equation*}
		\underset{T\rightarrow+\infty}{\lim}(2)^{1/T}\leq\underset{T\rightarrow+\infty}{\lim}\dfrac{1}{e^{1-\frac{1}{T}}\cdot \left(\sqrt{1-e^{-2}}\right)^{1-\frac{2}{T}}}=\dfrac{1}{e\cdot \left(\sqrt{1-e^{-2}}\right)}.
	\end{equation*}
	
	We work on  $(1)^{1/T}$,
	\begin{equation}
		\underset{T\rightarrow+\infty}{\lim}	(1)^{1/T}
		=	\underset{T\rightarrow+\infty}{\lim}\dfrac{4^{\frac{1}{T}}\cdot18\cdot 8^{1-\frac{2}{T}}\cdot e^{2-\frac{4}{T}}\cdot 18^{T-2}}{19^{T-1}e^{2(T-2)}}=0
	\end{equation}
	
	Now, we let $G^{(3)}_T=\left((1)\cdot(2)\cdot(3)\right)^{1/T}$.
	
	Therefore, when $n=\dfrac{J}{2\pi}$, 
	\begin{equation}
		\underset{T\rightarrow+\infty}{\lim}\left(P\left(\dfrac{1}{TJ}S_n>Kn^{-2}\right)\right)^{1/T}=0.
	\end{equation}
	
	Then we know that 
	
	\begin{equation*}
		\underset{T\rightarrow+\infty}{\lim}P\left(R(T,J)\geq K\right)^{1/T}
		\leq 3\cdot \underset{T\rightarrow+\infty}{\lim}\left(\max\left\{G^{(1)}_T, G^{(2)}_T,G^{(3)}_T\right\}\right)=0.
	\end{equation*}

\end{proof}

	\section{Damped wave equation with the noise F and its solution} \label{dampedwaveequationwiththenoiseF}
	Let $\phi\in\mathcal{C}^\infty(\mathbb{R})$ with $\phi '(0)=\phi ' (J)=0$. By multiplying (\ref{main system}) by $\phi(x)$ and integrating over both variables, we get
	\begin{equation}
		\begin{split}
			\int_{0}^{J}[\partial_tu(t.x)&-\partial_tu(0,x)]\phi(x)dx+\int_{0}^{J}[u(t.x)-u(0,x)]\phi(x)dx\\
			&=\int_{0}^{t}\int_{0}^{J}u(s,x)\phi''(x)dxds+\int_{0}^{t}\int_{0}^{J}\phi(x)F(dxds).\label{integ_form}
		\end{split}
	\end{equation}
	
	Let $\{\varphi_{n}\}_{n\in\mathbb{Z}}$ be a complete set of  eigenfunctions of the Laplacian $\Delta$ satisfying the Neumann boundary condition:
	\begin{equation*}
		\begin{split}
			\varphi_{n}(x)&=c_n\cos\big(\frac{n\pi}{J}x\big)\quad n\in\mathbb{Z}\\
			\Delta\varphi_n&=\lambda_n\varphi_n
		\end{split}
	\end{equation*}
	where 
	\[
	c_n=
	\begin{cases}
		\sqrt{\frac{2}{J}},\quad n\neq 0\\
		\sqrt{\frac{1}{J}}, \quad  n=0
	\end{cases}
	\]
	
	and
	\[
	\lambda_n=
	\begin{cases}
		-\left(\dfrac{n\pi}{J}\right)^2,&\quad n\neq 0\\
		0, &\quad  n=0.
	\end{cases}
	\]
	
	Considering the Fourier series of $u$, 
	\begin{equation}
		u(t,x)=\underset{n\in\mathbb{N}}{\sum}a_n(t)\varphi_{n}(x),\text{ where  } \Delta\varphi_{n}=\lambda_{n}\varphi_{n}.\label{ufourierseries}
	\end{equation}
	The series converges in $L^2\left([0,J]\times\Omega\right)$. The proof is in Appendix \ref{fourierseriesofmildsolutionu}

	Let $\phi$ in (\ref{integ_form}) be the eigenfunction $\varphi_n$. According to the fourier series of $u$, we have that for each $n\in\mathbb{N}$,
	\begin{equation*}
		\frac{d}{dt}a_n(t)-\frac{d}{dt}a_n(0)+a_n(t)-a_n(0)=\int_{0}^{t}\lambda_na_n(s)ds+\int_{0}^{t}\gamma_nW_n(ds).
	\end{equation*}
	
	
	Let
	\[\begin{cases}
		X^n_t&=a_n=\text{the position of the above stochastic oscillator}\\
		V^n_t&=\frac{d}{dt}a_n=\text{velocity}.
	\end{cases}
	\]
	We have the following stochastic differential equation.
	\[\begin{cases}
		dX^n_t&=V^n_t dt\\
		dV^n_t&=\big(-V^n_t+\lambda_n X^n_t\big)dt+\gamma_ndW_n(t).
	\end{cases}\]
	
	\subsection{When n is nonzero}
	
	When $n\neq 0$, we have the following stochastic differential equation,
	\begin{equation}
		d\begin{pmatrix} X^n_t\\ V^n_t \end{pmatrix}=\begin{pmatrix}0 & 1\\ \lambda_n & -1\end{pmatrix}\begin{pmatrix}X^n_t \\ V^n_t\end{pmatrix}dt+\begin{pmatrix}0 \\ \gamma_n \end{pmatrix}dW_n(t). \label{SDE}
	\end{equation} 
	
	Defining $M_n=\begin{pmatrix}
		0&1\\\lambda_n&-1
	\end{pmatrix}$,
	$\alpha_n=\begin{pmatrix}
		0 \\ \gamma_n
	\end{pmatrix}$, 
	and  $\vec{X}^n_t=\begin{pmatrix}X^n_t\\V^n_t\end{pmatrix}$, we have 
	\begin{equation}
		d\vec{X}^n_t=M_n\cdot\vec{X}^n_tdt+\gamma_ndW_n(t).\label{SDEnnot0}
	\end{equation}
	
	Multiplying $e^{-M_nt}$ to both sides of (\ref{SDEnnot0}), we have
	\begin{equation*}
		\begin{split}
			e^{-M_nt}d\vec{X}_n^t&=e^{-M_nt}\cdot M_n\cdot\vec{X}_n^tdt+e^{-M_nt}\gamma_ndW_n(t)\\
			e^{-M_nt}d\vec{X}_n^t-e^{-M_nt}\cdot M_n\cdot\vec{X}_n^tdt&=e^{-M_nt}\gamma_ndW_n(t)\\
			d(e^{-M_nt}\vec{X}_n^t)&=e^{-M_nt}\gamma_ndW_n(t).
		\end{split}
	\end{equation*}
	
	Then we have
	\begin{equation}
		\begin{split}
			&d\bigg(e^{-M_nt}\begin{pmatrix}X^n_t\\V^n_t\end{pmatrix}\bigg)=e^{-M_nt}\begin{pmatrix}0\\\gamma_n\end{pmatrix}dW_n(t)\\
			&\begin{pmatrix}X^n_t\\V^n_t\end{pmatrix}=e^{M_nt}\begin{pmatrix}x^n_0\\v^n_0\end{pmatrix}+e^{M_nt}\int_{0}^{t}e^{-M_ns}\begin{pmatrix}0\\\gamma_n\end{pmatrix}dW_n(s).\label{Xsoln}
		\end{split}
	\end{equation}
	where
	$x^n_0$ and $v^n_0$ are initial data of $X^n$ and $Y^n$ respectively.\\
	Note that since $tM_n$ and $(-s)M_n$ commute for all $t, s\in\mathbb{R}$, $e^{tM_n}\cdot e^{-sM_n}=e^{M_n(t-s)}$.\\
	Now we start to solve for $X^n$. \\
	
	Recall that 
	$M_n=\begin{pmatrix}
		0&1\\\lambda_n&-1
	\end{pmatrix}$,
	where $\lambda_n=-k_n^2=-\big(\frac{n\pi}{J}\big)^2$. 
	
	First, we need to find $e^{M_nt}$ using the eigenvalues and eigenvectors of the matrix $M_n$, where $M_n$ has characteristic polynomial:
	\begin{equation*}
		p(t)=\det(M_n-tI)=\bigg|\begin{matrix}-t&1\\-k_n^2&-1-t\end{matrix}\bigg|=t^2+t+k_n^2.
	\end{equation*}
	
	Setting $p(t)=0$, we find that two eigenvalues of $M_n$ are:
	\begin{equation*}
		c_{n}^{(1)}=\frac{-1+\sqrt{1-4k_n^2}}{2},\text{ and }
		c_{n}^{(2)}=\frac{-1-\sqrt{1-4k_n^2}}{2}.
	\end{equation*}

	There are three cases to discuss:
	\begin{enumerate}
		\item	$1-4k_{n}^2>0$, which is equivalent to $J>2n\pi$;
		\item	$1-4k_{n}^2=0$, which is equivalent to $J=2n\pi$; 
		\item	$1-4k_{n}^2<0$. which is equivalent to $J<2n\pi$. 
	\end{enumerate}
	
	\begin{enumerate}
		\item[\textbf{Case 1}] when $1-4k_n^2>0$, i.e $J>2n\pi$.\\
		
		In this case, $c^{(1)}_n$ and $c^{(2)}_n$ are two distinct eigenvalues. Their corresponding eigenvectors are $v^{(1)}_n=\begin{pmatrix}
			1 ,&c^{(1)}_n
		\end{pmatrix}^T$
		and
		$v^{(2)}_n=\begin{pmatrix}
			1 , &c^{(2)}_n
		\end{pmatrix}^T
		$. \\
		Let $V_n=\begin{pmatrix}
			v^{(1)}_n ,& v^{(2)}_n
		\end{pmatrix}$, and $D_n$ be the diagonal matrix with diagonal entries $c^{(1)}_n$ and $c^{(2)}_n$. Then we have
		\begin{equation*}
			\begin{split}
				M_n=V_nD_nV^{-1}_n,\qquad
				e^{M_nt}=V_ne^{D_nt}V^{-1}_n
			\end{split}
		\end{equation*}
		
		We get
		\begin{equation*}
			\begin{split}
				X^n_t=
				&-\frac{1}{w_n}\bigg(\bigg(\frac{-1-w_n}{2}e^{\frac{-1+w_n}{2}t}-\frac{-1+w_n}{2}e^{\frac{-1-w_n}{2}t}\bigg)x^n_0+\bigg(e^{\frac{-1-w_n}{2}t}-e^{\frac{-1+w_n}{2}t}\bigg)v^n_0\bigg)\\
				&-\frac{\gamma_{n}}{w_n}\int_{0}^{t}e^{\frac{-1-w_n}{2}(t-s)}-e^{\frac{-1+w_n}{2}(t-s)}dW_n(s).
			\end{split}
		\end{equation*}
		
		\item[\textbf{Case 2}] when $1-4k_n^2=0$, i.e $J=2n\pi$.
		
		$\lambda_{n}=-k_n^2=-\frac{1}{4}$. Then $M_n=\begin{pmatrix}0&1\\-\frac{1}{4}&-1\end{pmatrix}$ has two repeated eigenvalues, $c^{(1)}_n=c^{(2)}_n=-\frac{1}{2}$. We can write $M_n=PAP^{-1}$. \\
		
		We have
		\begin{equation*}
			A=\begin{pmatrix}
				-\frac{1}{2}&1\\0&-\frac{1}{2}
			\end{pmatrix},
			\quad
			P=\begin{pmatrix}
				1&0\\-\frac{1}{2}&1
			\end{pmatrix},
			\quad
			P^{-1}=\begin{pmatrix}
				1&0\\\frac{1}{2}&1
			\end{pmatrix}.
		\end{equation*}
		
		Thus, we have 
		\begin{equation*}
			e^{M_n}=Pe^AP^{-1}=Pe^{c^{(1)}_nI+N}P^{-1}=Pe^{c^{(1)}_nI}e^NP^{-1}
		\end{equation*}
		where $A=c^{(1)}_nI+N$ and $N=\begin{pmatrix}
			0&1\\0&0
		\end{pmatrix}$ is nilpotent with index 2. $N$ commutes with $c^{(1)}_nI$.\\
		
		Since $(tN)^2=\begin{pmatrix}
			0&0\\0&0
		\end{pmatrix}$ for all $t\in\mathbb{R}$,
		\begin{equation*}
			e^{tN}=I+\sum_{k=1}^{+\infty}\dfrac{(tN)^k}{k!}=I+tN=\begin{pmatrix}
				1&t\\0&1
			\end{pmatrix}.
		\end{equation*}
		
		Then $\forall t\in\mathbb{R}$,
		\begin{equation*}
			e^{M_nt}=\frac{1}{4}e^{-\frac{1}{2}t}\begin{pmatrix}4+2t & 4t \\ -t & 4-2t\end{pmatrix}.
		\end{equation*}
		
		Then we have
		\begin{equation*}
			X^n_t=\frac{1}{4}e^{-\frac{1}{2}t}[(4+2t)x_0^n+4tv_0^n]+\int_{0}^{t}e^{-\frac{1}{2}(t-s)}\gamma_{n}(t-s)dW_n(s).
		\end{equation*}

		\item[\textbf{Case 3}] when $1-4k_n^2<0$, i.e $0<J<2\pi n.$
		
		$M_n=\begin{pmatrix}
			0 & 1\\ \lambda_{n} & -1
		\end{pmatrix}$
		has two complex eigenvalues $c_n^{(1)}=\frac{-1+(\sqrt{4k_n^2-1})i}{2}$ and $c_n^{(2)}=\frac{-1-(\sqrt{4k_n^2-1})i}{2}$, where $\lambda_n=-k_n^2$.\\
		
		Let $\alpha_n=-\frac{1}{2}$ and $\omega_n=\frac{\sqrt{4k_n^2-1}}{2}$.  Then $M_n=P_nD_nP_n^{-1}$, where
		\begin{equation*}
			D_n=\begin{pmatrix}
				\alpha_n & \omega_n\\ -\omega_n & \alpha_n
			\end{pmatrix},
			\quad
			P_n=\begin{pmatrix}
				1 & 0 \\ \alpha_n & \omega_n
			\end{pmatrix},
			\quad \text{ and }
			P_n^{-1}=\begin{pmatrix}
				1 & 0 \\ -\frac{\alpha_n}{\omega_n} & \frac{1}{\omega_n}
			\end{pmatrix}.
		\end{equation*}
		
		We have
		\begin{equation*}
			e^{D_nt}=e^{\alpha_nt}\begin{pmatrix}
				\cos(\omega_nt)&\sin(\omega_nt)\\-\sin(\omega_nt)&\cos(\omega_nt)
			\end{pmatrix}.
		\end{equation*}
		
		Then 
		\begin{equation*}
			\begin{split}
				e^{M_nt}
				&=P_ne^{D_nt}P_n^-1\\
				&=\begin{pmatrix}1 & 0 \\ \alpha_n & \omega_n\end{pmatrix}e^{\alpha_nt}
				\begin{pmatrix}\cos\omega_nt & \sin\omega_nt\\ -\sin\omega_nt & \cos\omega_nt\end{pmatrix}
				\begin{pmatrix}1 & 0 \\ -\frac{\alpha_n}{\omega_n} & \frac{1}{\omega_n}\end{pmatrix}\\
				&=e^{-\frac{1}{2}t}\begin{pmatrix}
					\cos(\omega_nt)+\frac{1}{2\omega_n}\sin(\omega_nt) & \frac{1}{\omega_n}\sin(\omega_nt)\\
					-\omega_n\sin(\omega_nt)-\frac{1}{4\omega_n}\sin(\omega_nt) & -\frac{1}{2\omega_n}\sin(\omega_nt)+\cos(\omega_nt)
				\end{pmatrix}.
			\end{split}
		\end{equation*}
		
		From (\ref{Xsoln}),
		\begin{equation*}
			\begin{split}
				X^n_t=
				&e^{-\frac{1}{2}t}\big[\big(\cos(\omega_nt)+\frac{1}{2\omega_n}\sin(\omega_nt)\big)x_0^n+\frac{1}{\omega_n}\sin(\omega_nt)v_0^n\big]\\
				&+\int_{0}^{t}e^{-\frac{1}{2}(t-s)}\frac{\gamma_{n}}{\omega_n}\sin(\omega_n(t-s))dW_n(s).
			\end{split}
		\end{equation*}
		
	\end{enumerate}
	
	\subsection{n=0}
	
	When $n=0$, we have the following,
	\[\begin{cases}
		dX^0_t&=V^0_t dt\\
		dV^0_t&=-V^0_tdt+\gamma_0dW_0(t).
	\end{cases}\]
	
	Multiplying $e^t$ to the equation of $dV^0_t$, we get
	\begin{equation*}
		d\big(e^tV^0_t\big)=\gamma_0e^tdW_0(t).
	\end{equation*}
	
	Taking integrals on both sides with respect to time, we have
	\begin{equation*}
		\begin{split}
			V^0_t&=\gamma_0\int_{0}^{t}e^{s-t}dW_0(s)+v_0e^{-t}\\
			X^0_t&=\int_{0}^{t}V_s^0ds\\
			&=\gamma_0\int_{0}^{t}\int_{0}^{s}e^{\alpha-s}dW_0(\alpha)ds+v_0(1-e^{-t})+x_0
		\end{split}
	\end{equation*}
	where $x_0$ and $v_0$ are initial data. \\ 
	
	Since for every $s\in[0,t]$ where $t\in[0,T]$, $e^{\alpha-s}\leq1$ uniformly for $\alpha\in[0,s]$, we can apply stochastic Fubini to the integral term of $X_t^0$. In other words, 
	\begin{equation*}
		X^0_t=\gamma_0\int_{0}^{t}\int_{\alpha}^{t}e^{\alpha-s}dsdW_0(\alpha)+v_0(1-e^{-t})+x_0.
	\end{equation*}
	
	\section{Fourier series of the mild solution u.}\label{fourierseriesofmildsolutionu}
	Since 
	\begin{equation*}
		G_t^{\mathbb{R}}(x)=\dfrac{1}{2}e^{-t/2}\text{sgn}(t)\text{I}_0\left(\dfrac{1}{2}\sqrt{t^2-x^2}\right)\chi_{[-|t|, |t|]}(x)
	\end{equation*}
	is supported on $\left[-|t|, |t|\right]$ for $t\in[0,T]$, 
	\begin{equation*}
		\begin{split}
			G_t(x,y)
			&=\sum_{n\in\mathbb{Z}}\left(G_t^{\mathbb{R}}(y+x-2nJ)+G_t^{\mathbb{R}}(y-x-2nJ)\right)\\
			&=\sum_{|n|\leq M_T, n\in\mathbb{Z}}\left(G_t^{\mathbb{R}}(y+x-2nJ)+G_t^{\mathbb{R}}(y-x-2nJ)\right)
		\end{split}
	\end{equation*}
	where $M_T=\ceil[\bigg]{\dfrac{T}{J}}+1$.
	
	Then in the mild form, we can expand $G$ to $G^\mathbb{R}$,
	\begin{equation*}
		\begin{split}
			u(t,x)
			&=\int_{0}^{J}\partial_{t}G_t(x,y)u_0(y)dy+\int_{0}^{J}G_t(x,y)\big(\frac{1}{2}u_0(y)+u_1(y)\big)dy\\
			&\quad+\int_{0}^{t}\int_{0}^{J}G_{t-s}(x,y)F(dyds)\\
			&=\sum_{|n|\leq M_T, n\in\mathbb{Z}}\bigg[\int_{0}^{J}\partial_{t}G_t^{\mathbb{R}}(y+x-2nJ)u_0(y)dy\\
			&\quad+\int_{0}^{J}\partial_{t}G_t^{\mathbb{R}}(y-x-2nJ)u_0(y)dy\\
			&\quad+\int_{0}^{J}G_t^{\mathbb{R}}(y+x-2nJ)\big(\frac{1}{2}u_0(y)+u_1(y)\big)dy\\
			&\quad+\int_{0}^{J}G_t^{\mathbb{R}}(y-x-2nJ)\big(\frac{1}{2}u_0(y)+u_1(y)\big)dy\\
			&\quad+\int_{0}^{t}\int_{0}^{J}G_{t-s}^{\mathbb{R}}(y+x-2nJ)F(dyds)\\
			&\quad+\int_{0}^{t}\int_{0}^{J}G_{t-s}^{\mathbb{R}}(y-x-2nJ)F(dyds)\bigg].
		\end{split}
	\end{equation*}
	
	By Theorem 5.3 of \cite{Nu20}, Young's inequality, and $\Vert I_0\Vert<+\infty$, we know that 
	\begin{equation*}
		\E\left[\int_{0}^{J}|u(t,x)|^2dx\right]<+\infty,
	\end{equation*}
	then the Fourier series of $u$ in (\ref{ufourierseries}) converges in $L^2\left([0,J]\times\Omega\right)$.
	
	\section{Drift term}\label{driftterm}
	
	We add the drift term $a\varphi_1$ to the noise $F$. Let $v$ be the solution of the following model
	\begin{equation*}
		\begin{split}
			\partial_{t}^{2}v+\partial_{t}v&=\Delta v+a\varphi_1\\
			v(0,x)=v_{0}(x), &\quad \partial_{t}v(0,x)=v_{1}(x) \quad (x, t)\in[0,J]\times\mathbb{R}_+\\
			\partial_{x}v(t,0)&=\partial_{x}(t,J)=0
		\end{split}
	\end{equation*}
	where $\varphi_1=\sqrt{\dfrac{2}{J}}\cos\left(\dfrac{\pi}{J}x\right)$.
	
	We can write $v(t,x)=d(t)\varphi_1(x)$ where $d$ is a function depending on $t$. 
	
	To solve the above model, we apply the same technique as in Appendix \ref{dampedwaveequationwiththenoiseF}. Considering the process in the future, we get
	\begin{equation*}
		d(t)=\begin{cases}
			\frac{a}{\omega_1}\int_{-\infty}^{t}e^{\frac{-1+\omega_1}{2}(t-s)}-e^{\frac{-1-\omega_1}{2}(t-s)}ds & J>2\pi \\
			a\int_{-\infty}^{t}e^{-\frac{1}{2}(t-s)}\gamma_{1}(t-s)ds & J=2\pi \\
			\frac{a}{\omega_1}\int_{-\infty}^{t}e^{-\frac{1}{2}(t-s)}\frac{\gamma_{1}}{\omega_1}\sin(\omega_1(t-s))ds & 0<J<2\pi .
		\end{cases}
	\end{equation*}

	That is 
	\begin{equation}
		d(t)=\begin{cases}
			\dfrac{a}{\omega_1}\left(\dfrac{2}{1-\omega_1}-\dfrac{2}{1+\omega_1}\right) & J>2\pi \\
			2a & J=2\pi \\
			\dfrac{4a}{5\omega_1} & 0<J<2\pi .
		\end{cases}\label{C'driftconstant}
	\end{equation}
	
	\section{Proof of Lemma \ref{lowerboundlemma3.1}}\label{lowerboundofsigmax1x2square}
	
	\begin{proof}
		The proof basically follows the proof of Lemma 2.7 of \cite{MN22}. The difference is that we need to set $x=\frac{x_1+x_2}{2J}$ and $h=\frac{x_1-x_2}{J}$.
		
		Let $\mathcal{U}(x_1, x_2)=\tilde{U}(0,x_1)-\tilde{U}(0,x_2)$. then
		\begin{equation*}
			\Var\big[\mathcal{U}(x_1, x_2)\big]=\E\big[\mathcal{U}(x_1, x_2)^2\big].
		\end{equation*}
		
		We have 
		\begin{equation*}
			\begin{split}
				\mathcal{U}(x_1, x_2)^2
				&=\bigg(\sum_{n\neq 0}\tilde{a}_n(0)\big(\varphi_n(x_1)-\varphi_n(x_2)\big)\bigg)^2\\
				&=\sum_{n, m\neq 0}\tilde{a}_n(0)\tilde{a}_m(0)\big(\varphi_n(x_1)-\varphi_n(x_2)\big)\big(\varphi_m(x_1)-\varphi_m(x_2)\big)\\
				&=\sum_{n\neq0}\tilde{a}_n(0)^2\big(\varphi_n(x_1)-\varphi_n(x_2)\big)^2\\
				&+2\sum_{m>n>0}\tilde{a}_n(0)\tilde{a}_m(0)\big(\varphi_n(x_1)-\varphi_n(x_2)\big)\big(\varphi_m(x_1)-\varphi_m(x_2)\big).
			\end{split}
		\end{equation*}
		
		Then we take expectation on both sides
		
		\begin{equation*}
			\begin{split}
				\E\bigg[\mathcal{U}(x_1, x_2)^2\bigg]
				&=\sum_{n\neq0}\E\bigg[\tilde{a}_n(0)^2\big(\varphi_n(x_1)-\varphi_n(x_2)\big)^2\bigg]\\
				&=\sum_{n\neq0}\E\bigg[\tilde{a}_n(0)^2\bigg]\big(\varphi_n(x_1)-\varphi_n(x_2)\big)^2\\
				&+2\sum_{m>n>0}\E\bigg[\tilde{a}_n(0)\tilde{a}_m(0)\bigg]\big(\varphi_n(x_1)-\varphi_n(x_2)\big)\big(\varphi_m(x_1)-\varphi_m(x_2)\big)\\
				&=\sum_{n\neq0}\E\bigg[\tilde{a}_n(0)^2\bigg]\big(\varphi_n(x_1)-\varphi_n(x_2)\big)^2
			\end{split}
		\end{equation*}
		where
		\begin{enumerate}
			\item[$\cdot$] When $J<2\pi n$, $\E[\mid\tilde{a}_n(t)\mid^2]=\frac{\gamma_n^2}{2(1+4\omega_n^2)}$.
			\item[$\cdot$] When $J=2\pi n$, $\E[\mid\tilde{a}_n(t)\mid^2]=2\gamma_n^2$.
			\item[$\cdot$] When $J>2\pi n$, $\E[\mid\tilde{a}_n(t)\mid^2]=\frac{2\gamma_n^2}{\omega_n(1-\omega_n^2)}$.
		\end{enumerate}
		
		Since $\gamma_n\rightarrow0$ as $n\rightarrow\infty$ and $\omega_n=\sqrt{\left|1-4\big(\frac{n\pi}{J}\big)^2\right|}$, for all $n$, $\E[\mid\tilde{a}_n(t)\mid^2]$ is bounded for all $t\in\mathbb{R}$. 
		
		Let $c_n=\E[\tilde{a}_n(0)^2]$, 
		we recall $\varphi_n(x)=\sqrt{\dfrac{2}{J}}\cos\left(\dfrac{n\pi}{J}x\right)$, then
		\begin{equation*}
			\begin{split}
				\sigma^2&=\E\bigg[\mathcal{U}(x_1, x_2)^2\bigg]=\sum_{n\neq0}\E\bigg[\tilde{a}_n(0)^2\bigg]\big(\varphi_n(x_1)-\varphi_n(x_2)\big)^2\\
				&=\dfrac{4}{J}\sum_{n\neq0}c_n\left[\cos\left(\dfrac{n\pi}{J}x_1\right)-\cos\left(\dfrac{n\pi}{J}x_2\right)\right]^2.
			\end{split}
		\end{equation*}
		
		Since $\cos(a)-\cos(b)=-2\sin\left(\frac{a-b}{2}\right)\sin\left(\frac{a+b}{2}\right)$, we have
		\begin{equation*}
			\sigma^2=\dfrac{16}{J}\sum_{n=1}^{+\infty}c_n\sin^2\left(\dfrac{n\pi}{2J}(x_1-x_2)\right)\sin^2\left(\dfrac{n\pi}{2J}(x_1+x_2)\right).
		\end{equation*}
		
		Let $x=\dfrac{x_1+x_2}{2J}$ and $h=\dfrac{x_1-x_2}{J}$. By symmetry, we assume $x\in\left[0,\dfrac{1}{2}\right]$.\\
		It suffices to prove the estimate for $h<\delta_0$ where $\delta_0>0$ is small.
		
		Since $c_n>0, \forall n$, for any $N>0$,
		\begin{equation*}
			\begin{split}
				\sigma^2
				&=\dfrac{16}{J}\sum_{n=1}^{+\infty}c_n\sin^2\left(\dfrac{n\pi}{2}h\right)\sin^2\left(n\pi x\right)\\
				&\geq\dfrac{16}{J}\sum_{n=1}^{N}c_n\sin^2\left(\dfrac{n\pi}{2}h\right)\sin^2\left(n\pi x\right).
			\end{split}
		\end{equation*}
		
		Note that $x_2=\left(x-\dfrac{h}{2}\right)J$ and $x_2\geq0$, we have $x\geq\dfrac{h}{2}$.
		
		Let $\delta_1>0$ be a small number, and 
		\begin{equation*}
			N=[2h^{-1}(1-\delta_1)], \text{ where }[\cdot] \text{ represents the greatest integer function.}
		\end{equation*}
		
		Then we have
		\begin{equation}
			\pi(1-\delta_1)-1\leq \pi N2^{-1}h\leq \pi(1-\delta_1). \label{*}
		\end{equation}
		
		Given any $n$ such that $1\leq n\leq N$, we have
		\begin{equation*}
			\sin\left(\dfrac{\pi}{2}nh\right)\geq cnh
		\end{equation*}
		where $c$ is a constant.
		
		Let $m_1=16\underset{1\leq n\leq N}{\min}c_n$, then 
		\begin{equation*}
			\begin{split}
				\sigma^2
				&\geq cm_1h^2\frac{1}{J}\sum_{n=1}^{N}\sin^2(n\pi x)\\
				&=cm_1h^2\frac{1}{J}\left[2N+1-\dfrac{\sin((2N+1)\pi x)}{\sin(\pi x)}\right].
			\end{split}
		\end{equation*}
		
		We want to show $2N+1-\dfrac{\sin((2N+1)\pi x)}{\sin(\pi x)}$ is of order $N$. So we need to show that for some small number $\delta_2>0$,
		\begin{equation}
			\dfrac{\sin((2N+1)\pi x)}{\sin(\pi x)}\leq 2N(1-\delta_2). \label{**}
		\end{equation}
		
		Let $\delta_3>0$, since $\dfrac{h}{2}\leq x\leq \dfrac{1}{2}$ and $h<\delta_0$, we can choose $\delta_0$ small enough such that
		\begin{equation*}
			\sin(\pi x)\geq\sin\left(\frac{\pi}{2}h\right)\geq\frac{\pi}{2}h(1-\delta_3).
		\end{equation*}
		
		By (\ref{*}), we have
		\begin{equation*}
			\begin{split}
				\dfrac{\sin((2N+1)\pi x)}{\sin(\pi x)}
				&\leq\dfrac{1}{\frac{\pi}{2}h(1-\delta_3)}\\
				&=N\dfrac{1}{(N\pi2^{-1}h)(1-\delta_3)}\\
				&\leq N\dfrac{1}{[\pi(1-\delta_1)-1](1-\delta_3)}.
			\end{split}
		\end{equation*}
		
		The above inequalitiy verifies (\ref{**}) provided $\delta_1$, $\delta_2$ and $\delta_3$ are small enough.  So we have
		\begin{equation*}
			\begin{split}
				\sigma^2
				&\geq cm_1h^2\frac{1}{J}[2N+1-\dfrac{\sin((2N+1)\pi x)}{\sin(\pi x)}]\\
				&\geq cm_1h^2N\delta_2\frac{1}{J}\\
				&\geq 2cm_1h\delta_2\frac{1}{J}, \quad\text{ by } (\ref{*})\\
				&=2\tilde{c}\dfrac{|x_1-x_2|}{J^2}\delta_2, \quad\text{ for all }|x_1-x_2|\leq \delta_0
			\end{split}
		\end{equation*}
		where $\tilde{c}(J)=cm_1$.
		
	\end{proof}
	
	\section{Proof of Lemma 4.1} \label{proofofintegraleexpbrownianmotion}
	
	\begin{proof}
		From (\ref{Btildeupperbound}), we have the following decomposition of the integral
		$\int_{0}^{T}e^{-at}\tilde{B}_{e^{bt}}^2dt$,
		\begin{equation*}
			\begin{split}
				\int_{0}^{T}&e^{-at}\tilde{B}_{e^{bt}}^2dt
				=\int_{0}^{1}e^{-at}\tilde{B}_{e^{bt}}^2dt+\int_{1}^{2}e^{-at}\tilde{B}_{e^{bt}}^2dt+\cdots+\int_{T-1}^{T}e^{-at}\tilde{B}_{e^{bt}}^2dt\\
				&\leq\int_{0}^{1}e^{-at}\tilde{B}^2_{e^{bt}}dt
				+\sum_{k=1}^{T-1}\int_{k}^{k+1}e^{-at}\left(\sum_{m=1}^{k-1}3^{k-m}\tilde{B}_{e^{mb},e^{(l+m)b}}+3^k\tilde{B}_{e^b}+\tilde{B}_{e^{kb},e^{bt}}\right)^2dt
			\end{split}
		\end{equation*}
		where the sum over $m$ is zero when $k=1$.
		
		Then we apply the Cauchy-Schwarz inequality to the integrals from above,
		\begin{equation*}
			\begin{split}
				\int_{0}^{T}&e^{-at}\tilde{B}_{e^{bt}}^2dt
				\leq \int_{0}^{1}e^{-at}\tilde{B}_{e^{bt}}^2dt\\
				&\hspace{2cm} +\sum_{k=1}^{T-1}\int_{k}^{k+1}e^{-at}
				\Bigg[\sum_{m=1}^{k-1}2^{k+1-m}\left(3^{k-m}\tilde{B}_{e^{mb},e^{(l+m)b}}\right)^2\\
				&\hspace{5cm}
				+2^k\left(3^k\tilde{B}_{e^b}\right)^2+2\left(\tilde{B}_{e^{kb},e^{bt}}\right)^2
				\Bigg]dt
			\end{split}
		\end{equation*}
		where the sum over $m$ is zero when $k=1$.
		
		We rearrange the order of the above integrals by grouping with similar integrands. 
		
		Then, we have
		\begin{equation*}
			\begin{split}
				\int_{0}^{T}e^{-at}\tilde{B}^2_{e^{bt}}dt
				&=\int_{0}^{1}e^{-at}\tilde{B}^2_{e^{bt}}dt+\tilde{B}_{e^{b}}^2\sum_{k=1}^{T-1}\int_{k}^{k+1}2^k\cdot3^{2k}e^{-at}dt\\
				&\quad +\sum_{m=1}^{T-2}\bigg[\bigg(\sum_{k=m+1}^{T-1}\int_{k}^{k+1}2^{k-m+1}\cdot3^{2(k-m)}e^{-at}\tilde{B}^2_{e^{bm},e^{b(m+1)}}dt\bigg)\\
				&\hspace{4.5em}
				+2\int_{m}^{m+1}e^{-at}\tilde{B}^2_{e^{bm},e^{bt}}dt\bigg].
			\end{split}
		\end{equation*}
		
		Simplifying constant coefficients, the integral becomes
		\begin{align*}
			\int_{0}^{T}e^{-at}\tilde{B}^2_{e^{bt}}dt
			&=
			\int_{0}^{1}e^{-at}\tilde{B}^2_{e^{bt}}dt+\tilde{B}_{e^{b}}^2
			\sum_{k=1}^{T-1}18^k\int_{k}^{k+1}e^{-at}dt\\
			&\quad +\sum_{m=1}^{T-2}
			\bigg[
			\bigg(
			\sum_{k=m+1}^{T-1}2\cdot18^{k-m}\int_{k}^{k+1}
			e^{-at}\tilde{B}^2_{e^{bm},e^{b(m+1)}}dt
			\bigg)\\
			&\hspace{4.5em}
			+2\int_{m}^{m+1}e^{-at}\tilde{B}^2_{e^{bm},e^{bt}}dt
			\bigg].
		\end{align*}
		
	\end{proof}
	
	\section{Proof of Lemma 4.2} \label{integraloftailprobability}
	
	\begin{proof}
		First, we start by changing the variable. Let $y=\log x$, then $dy=\dfrac{1}{x}dx$. That is $e^ydy=dx$. Then 
		\begin{equation*}
			\begin{split}
				\int_{1}^{+\infty}P\left(X\geq\sqrt{\dfrac{\log x}{\gamma}}\right)dx
				&=\int_{0}^{+\infty}P\left(X\geq\sqrt{\dfrac{y}{\gamma}}\right)e^ydy\\
				&=\int_{0}^{+\infty}\left(\int_{\sqrt{\frac{y}{\gamma}}}^{+\infty}\dfrac{1}{\sqrt{2\pi \sigma^2}}e^{-\frac{z^2}{2\sigma^2}}dz\right)e^ydy.
			\end{split}
		\end{equation*}
		From \cite{Dubook}, we have
		
		\begin{align*}
		\int_{0}^{+\infty}\left(
		\int_{\sqrt{\frac{y}{\gamma}}}^{+\infty}
		\dfrac{1}{\sqrt{2\pi \sigma^2}}e^{-\frac{z^2}{2\sigma^2}}dz
		\right)e^ydy
		&\leq
		\int_{0}^{+\infty}\int_{\sqrt{\frac{y}{\gamma}}}^{+\infty}
		\frac{z}{\sqrt{\frac{y}{\gamma}}}\cdot\frac{1}{\sqrt{2\pi\sigma^2}}
		e^{-\frac{z^2}{2\sigma^2}}dz
		e^ydy.
		\end{align*}
		
		Changing variables by setting $p=\frac{z^2}{2\sigma^2}$, the inner integral becomes
		\[
		\frac{\sigma^2}{\sqrt{2\pi\sigma^2}\cdot\sqrt{\frac{y}{\gamma}}}
		\int_{\frac{y}{2\gamma\sigma^2}}^{+\infty}e^{-p}dp
		=
		\frac{\sigma^2}{\sqrt{2\pi\sigma^2}\cdot\sqrt{\frac{y}{\gamma}}}
		e^{-\frac{y}{2\gamma\sigma^2}}.
		\]
		
		Then we have the following
		
		\begin{equation*}
			\begin{split}
				\int_{1}^{+\infty}P\left(X\geq\sqrt{\dfrac{\log x}{\gamma}}\right)dx
				\leq
				\dfrac{1}{\sqrt{2\pi\sigma^2}}\int_{0}^{+\infty}\dfrac{\sqrt{\gamma}\sigma^2}{\sqrt{y}}\exp\left(\left(1-\dfrac{1}{2\gamma\sigma^2}\right)y\right)dy.
			\end{split}
		\end{equation*}
		
		Let $u=\sqrt{y}$, 
		\begin{equation*}
			\int_{1}^{+\infty}P\left(X\geq\sqrt{\dfrac{\log x}{\gamma}}\right)dx
			\leq\dfrac{2\sigma^2\sqrt{\gamma}}{\sqrt{2\pi\sigma^2}}\int_{0}^{+\infty}\exp\left(\left(1-\dfrac{1}{2\gamma\sigma^2}\right)u^2\right)du.
		\end{equation*}
		Since $1-\dfrac{1}{2\gamma\sigma^2}<0$, the above integral converges. And for all $b>0$, we have  $\int_{0}^{+\infty}e^{-bx^2}=\dfrac{\sqrt{\pi}}{2\sqrt{b}}$. Thus,
		\begin{equation*}
			\begin{split}
				\int_{1}^{+\infty}P\left(X\geq\sqrt{\dfrac{\log x}{\gamma}}\right)dx
				&\leq\dfrac{2\sigma^2\sqrt{\gamma}}{\sqrt{2\pi\sigma^2}}\cdot\dfrac{\sqrt{\pi}}{2\sqrt{\frac{1}{2\gamma\sigma^2}-1}}\\
				&=\dfrac{\sigma^2\gamma}{\sqrt{1-2\gamma\sigma^2}}.
			\end{split}
		\end{equation*}
	\end{proof}
	
	\section{Proof of Lemma 4.3 } \label{exp9/10t}
	
	\begin{proof}
		Let $M(t)=\sum_{n=3}^{\infty}\dfrac{\left(\frac{9}{10}t\right)^n}{n!}$ and $f(t)= e^{\frac{9}{10}t}-t^2-1$, then
		\begin{equation}
			\begin{split}
				f(t)
				&=\sum_{n=0}^{\infty}\dfrac{\left(\frac{9}{10}t\right)^n}{n!}-t^2-1\\
				&=1+\dfrac{9}{10}t+\dfrac{1}{2}\cdot\dfrac{81}{100}t^2+M(t)-t^2-1.\label{taylor}
			\end{split}
		\end{equation}
		We simplify (\ref{taylor}),
		\begin{equation*}
			\begin{split}
				f(t)
				&=\dfrac{9}{10}t+\left(\dfrac{81}{200}-1\right)t^2+M(t)\\
				&=\dfrac{t}{200}(180-119t)+M(t).
			\end{split}
		\end{equation*}
		When $180-119t\geq0$, i.e. $t\leq \dfrac{180}{119}$, we have $f(t)\geq0$.\\
		Then we are left to show when $t\geq\dfrac{180}{119}$, $f(t)\geq0$ is also true.\\
		
		We know that $e^{\frac{9}{10}t}$ and $t^2+1$ are monotonic increasing functions when $t\geq0$, then 
		at $t=\dfrac{180}{119}$, 
		\begin{equation*}
			f\left(\dfrac{180}{119}\right)=e^{\frac{9}{10}\cdot\frac{180}{119}}-\left(\left(\dfrac{180}{119}\right)^2+1\right)\approx3.901-3.288>0.
		\end{equation*}
		The first derivative is
		\begin{equation*}
			f'\left(\dfrac{180}{119}\right)=\dfrac{9}{10}e^{\frac{9}{10}\cdot\frac{180}{119}}-2\cdot \dfrac{180}{119}\approx 3.511-3.0252>0.
		\end{equation*}
		The second derivative is
		\begin{equation*}
			f''\left(\dfrac{180}{119}\right)=\left(\dfrac{9}{10}\right)^2e^{\frac{9}{10}\cdot\frac{180}{119}}-2\approx 3.160-2>0.
		\end{equation*}
		Since $f''$ is an increasing function,  when $t\geq\dfrac{180}{119}$, $f''(t)\geq0$.\\
		Since $f'\left(\dfrac{180}{119}\right)$, $f'(t)\geq0$  $\forall t\geq\dfrac{180}{119}$.
		Then we know that $f(t)\geq 0$ for $t\in\left[\dfrac{180}{119},+\infty\right]$.\\
		
		We finished the proof.
	\end{proof}

	\section{Noise F}
	
	Let $\{W_n\}$ be independent and identical distributed white noise in time and $\{\gamma_n\}_{n\in\mathbb{N}}$ is a collection of real numbers such that 
	\begin{equation*}
		\sum_{n\in\mathbb{N}}\gamma_n^2<+\infty \quad \text{ and }\quad \gamma_n^2\leq\dfrac{c}{n^\alpha} \forall n\in\mathbb{N}
	\end{equation*}
	where $c$ and $\alpha$ are positive constants. 
	
	Intuitively, we have
	\begin{equation*}
		\begin{split}
			\Cov[\dot{F}&(t,x),\dot{F}(s,y)]\\
			&=\E\left[\left(\sum_{n\in\mathbb{N}}\gamma_n\dot{W}_n(t)\varphi_n(x)\right)\cdot\left(\sum_{m\in\mathbb{N}}\gamma_m\dot{W}_m(s)\varphi_m(y)\right)\right]\\
			&=\sum_{n,m\in\mathbb{N}}\gamma_n\gamma_m\E\left[\dot{W}_n(t)\dot{W}_m(t)\right]\varphi_n(x)\varphi_m(y)\\
			&=\delta(t-s)\sum_{n\in\mathbb{N}}\gamma_n^2\varphi_n(x)\varphi_n(y).
		\end{split}
	\end{equation*}

\bibliography{bibdatabase}
\bibliographystyle{plain}

\end{document}